\documentclass[11pt,leqno]{amsart}
\usepackage{amssymb,amsmath,amsthm,amsfonts,color}
\usepackage{graphicx}

\usepackage{hyperref}

\usepackage[dvipsnames]{xcolor}

\DeclareMathOperator{\spa}{span}

\DeclareMathOperator{\diam}{diam}

 \DeclareMathOperator{\SRA}{SRA}

\newcommand{\field}[1]{\mathbb{#1}}

\newcommand{\N}{\field{N}}                      
\newcommand{\R}{\field{R}}                      
\newcommand{\C}{\field{C}}                      
\newcommand{\Heis}{\field{H}}                    

\newcommand{\eps}{\epsilon}
\newcommand{\vareps}{\varepsilon}

\newcommand{\snowflake}{{\scriptstyle{snow}}}
\newcommand{\sra}{{\mbox{SRA}}}

\newcommand{\cD}{{\mathcal D}}

\newcommand{\cR}{{\mathcal R}}
\newcommand{\cS}{{\mathcal S}}

\newcommand{\bolde}{{\mathbf e}}
\newcommand{\bI}{{\mathbf I}}

\newcommand{\bp}{{\mathbf p}}

\newcommand{\bw}{{\mathbf w}}

\newcommand{\abs}[1]{\lvert #1 \rvert}

\def\Barint_#1{\mathchoice
          {\mathop{\vrule width 6pt height 3 pt depth -2.5pt
                  \kern -8pt \intop}\nolimits_{#1}}%
          {\mathop{\vrule width 5pt height 3 pt depth -2.6pt
                  \kern -6pt \intop}\nolimits_{#1}}%
          {\mathop{\vrule width 5pt height 3 pt depth -2.6pt
                  \kern -6pt \intop}\nolimits_{#1}}%
          {\mathop{\vrule width 5pt height 3 pt depth -2.6pt
                  \kern -6pt \intop}\nolimits_{#1}}}

\theoremstyle{plain}
\newtheorem{theorem}{Theorem}
\newtheorem{corollary}[theorem]{Corollary}
\newtheorem{lemma}[theorem]{Lemma}
\newtheorem{proposition}[theorem]{Proposition}

\theoremstyle{definition}
\newtheorem{definition}[theorem]{Definition}
\newtheorem{example}[theorem]{Example}
\newtheorem{remark}[theorem]{Remark}

\newtheorem{question}[theorem]{Question}

\numberwithin{theorem}{section} 
\numberwithin{equation}{section}

\title[Small rough angles and rough self-contracting curves]{Metric spaces with small rough angles and the rectifiability of rough self-contracting curves}
\author{Estibalitz Durand Cartagena and Jeremy T. Tyson}
\address{EDC: Departamento de Matem\'atica Aplicada, ETS de Ingenieros Industriales, UNED, 28040 Madrid, Spain}
\email{{\tt edurand@ind.uned.es}}
\address{JTT: University of Illinois at Urbana-Champaign, Department of Mathematics, 1409 West Green St.\ Urbana, IL 61801 USA}
\email{{\tt tyson@illinois.edu}}
\date{\today}
\thanks{{\bf Acknowledgements:} The research for this work was conducted while EDC was visiting the Department of Mathematics at University of Illinois at Urbana-Champaign during the spring semester of 2024. She gratefully acknowledges financial support for this research from the Fulbright Program, which is sponsored by the U.S. Department of State and the U.S.-Spain Fulbright Commission. EDC is partially supported by grant PID2022-138758NB-I00 (Spain). JTT acknowledges support from the Simons Foundation under grant \#852888.
In addition, this material is based upon work supported by and while JTT was serving as a Program Director at the U.S. National Science Foundation. Any opinion, findings, and conclusions or recommendations expressed in this material are those of the authors and do not necessarily reflect the views of the National Science Foundation.}

\begin{document}
\maketitle

\begin{abstract}
The {\it small rough angle} ($\SRA$) condition, introduced by Zolotov in {\tt arXiv:1804.00234}, captures the idea that all angles formed by triples of points in a metric space are small. In the first part of the paper, we develop the theory of metric spaces $(X,d)$ satisfying the $\SRA(\alpha)$ condition for some $\alpha<1$. Given a metric space $(X,d)$ and $0<\alpha<1$, the space $(X,d^\alpha)$ satisfies the $\SRA(2^\alpha-1)$ condition. We prove a quantitative converse up to bi-Lipschitz change of the metric. We also consider metric spaces which are $\SRA(\alpha)$ free (there exists a uniform upper bound on the cardinality of any $\SRA(\alpha)$ subset) or $\SRA(\alpha)$ full (there exists an infinite $\SRA(\alpha)$ subset). Examples of $\SRA$ free spaces include Euclidean spaces, finite-dimensional Alexandrov spaces of non-negative curvature, and Cayley graphs of virtually abelian groups; examples of $\SRA$ full spaces include the sub-Riemannian Heisenberg group, Laakso graphs, and Hilbert space. We study the existence or nonexistence of $\SRA(\eps)$ subsets for $0<\eps<2^\alpha-1$ in metric spaces $(X,d^\alpha)$ for $0<\alpha<1$.

In the second part of the paper, we apply the theory of metric spaces with small rough angles to study the rectifiability of roughly self-contracting curves. In the Euclidean setting, this question was studied by Daniilidis, Deville, and the first author using direct geometric methods. We show that in any $\SRA(\alpha)$ free metric space $(X,d)$, there exists $\lambda_0 = \lambda_0(\alpha)>0$ so that any bounded roughly $\lambda$-self-contracting curve in $X$, $\lambda \le \lambda_0$, is rectifiable. The proof is a generalization and extension of an argument due to Zolotov, who treated the case $\lambda=0$, i.e., the rectifiability of self-contracting curves in $\SRA$ free spaces.
\end{abstract}

\tableofcontents

\section{Introduction}

\subsection{Overview}

Our goals in this paper are twofold. First, we provide a gentle introduction to the theory of metric spaces with {\it small rough angles}, a condition introduced by Zolotov in \cite{Z}. The $\SRA(\alpha)$ condition, $0 \le \alpha < 1$, on a metric space $(X,d)$ is a strengthened form of the triangle inequality which implies, in particular, that all metric angles formed by triples of points in $X$ are bounded away from $\pi$ (quantitatively in terms of $\alpha$). In particular, we discuss the relationship between this condition and the snowflaking operation $(X,d) \mapsto (X,d^\eps)$, $0<\eps<1$. We go on to consider metric spaces $(X,d)$ which are either {\it $\SRA(\alpha)$ free} or {\it $\SRA(\alpha)$ full} for some $\alpha < 1$. The former condition requires the existence of a uniform upper bound on the cardinality of any $\SRA(\alpha)$ subset of $X$, while the latter condition asserts the existence of an infinite $\SRA(\alpha)$ subset. These conditions were also introduced by Zolotov in \cite{Z}, see also \cite{LOZ}. Euclidean space $\R^n$ is $\SRA(\alpha)$ free for any $n \in \N$ and any $\alpha<1$; this follows from a classical theorem of geometric combinatorics due to Erd\"os and F\"uredi. We present various examples of spaces with these properties, and we study the existence of large $\SRA(\alpha)$ subsets in snowflaked metric spaces.

Next, we establish the rectifiability of a class of metrically defined curves in $\SRA(\alpha)$ free metric spaces, known as {\it roughly self-contracting curves}. Self-contracting curves have been introduced by Daniilidis, Ley and Sabourau \cite{DLS} in connection with the theory of gradient flows for convex potential functions, and the rectifiability of such curves has been established by various authors in spaces of increasing generality. By now it is known that bounded, self-contracting curves are rectifiable in arbitrary Riemannian manifolds, and also in metric spaces satisfying suitable synthetic curvature bounds (\`a la Alexandrov). In \cite{Z}, Zolotov establishes the rectifiability of bounded, self-contracting curves in any metric space $(X,d)$ which is $\SRA(\alpha)$ free for some $\tfrac12 < \alpha < 1$. For each $\lambda \in [-1,1]$, the class of roughly $\lambda$-self-contracting curves is defined via a metric inequality, similar to the $\SRA(\alpha)$ condition but imposed only on ordered triples of points chosen along the curve. These classes interpolate between the class of geodesic curves and all possible curves, in the following sense: roughly $(-1)$-self-contracting curves are precisely the geodesics, while any curve is roughly $1$-self-contracting. A curve is roughly $0$-self-contracting if and only if it is self-contracting. We extend Zolotov's result to cover roughly $\lambda$-self-contracting curves for some positive choices of $\lambda$.

An important ancillary aim of this paper is to advertise $\SRA$ free metric spaces as natural objects for study within the framework of analysis in metric spaces. In this paper, we highlight connections to topics such as bi-Lipschitz embeddability, rectifiability, and the snowflaking operation. We also aim to increase the visibility of the results obtained by Zolotov in \cite{Z}, and especially to showcase the interesting methodology in his proofs. The second part of this paper is inspired heavily by the results and techniques in \cite{Z}.

\subsection{Statement of main results}

For $0\le \alpha \le 1$, a metric space $(X,d)$ is said to satisfy the {\em $\SRA(\alpha)$ condition}\footnote{The acronym $\SRA$ stands for {\em small rough angles}.} if $d(x,y) \le \max \{ d(x,z) + \alpha d(z,y) , \alpha d(x,z) + d(z,y) \}$ for all $x,y,z \in X$. The class of $\SRA(\alpha)$ spaces increases as $\alpha$ increases, with the $\SRA(0)$ condition coinciding with the well-known concept of {\it ultrametric} and the $\SRA(1)$ condition coinciding with the usual triangle inequality. Geometrically, the $\SRA(\alpha)$ condition on a metric space $(X,d)$ implies an upper bound on all metric angles formed by triples of points in $X$ (Remark \ref{rem:SRA-angles}). 

An alternate one-parameter family of conditions interpolating between the class of ultrametric spaces and the class of all metric spaces was studied by the second author and Wu in \cite{TW}. For $1<p<\infty$, $(X,d)$ is said to be an {\it $L^p$ metric space} if $d(x,y) \le (d(x,z)^p + d(z,y)^p)^{1/p}$ for all $x,y,z \in X$. When $p \to \infty$, the $L^p$ metric condition limits to the usual ultrametric condition $d(x,y) \le \max \{ d(x,z), d(z,y) \}$. For any $0<\eps<1$, the expression $d^\eps$ defines an $L^p$ metric on $X$, $p=\eps^{-1}$, and the transformation $(X,d) \to (X,d^\eps)$ is commonly referred to as the {\it snowflaking transformation} in the literature. In this paper, following the terminology in \cite{TW}, we call a metric space $(X,d)$ a {\it $p$-snowflake} if $d$ is bi-Lipschitz equivalent to an $L^p$ metric $d'$ on $X$. Building on observations by Le Donne, Rajala, and Walsberg \cite{LRW} and Zolotov \cite{Z}, we note that the metric space $(X,d^\alpha)$, $0<\alpha<1$, satisfies the $\SRA(2^{\alpha}-1)$ condition for any metric space $(X,d)$. The value $2^\alpha-1$ is best possible for such a conclusion with no further restrictions on $X$ (Example \ref{ex:arithmetic}). We devote some space to consideration of the converse assertion, namely, whether or not the validity of an $\SRA(\beta)$ condition implies that the underlying metric is an $L^p$ metric for some $p>1$. Such a conclusion is true for any finite set, but we give examples of infinite metric spaces for which such a converse statement is false. However, the situation is clarified if we allow for a bi-Lipschitz change of the metric. This is the content of our first main theorem.

\begin{theorem}\label{main-prop}
For any metric space $(X,d)$, the following conditions are quantitatively equivalent:
\begin{itemize}
\item[(i)] There exists a metric $d'$ on $X$ bi-Lipschitz equivalent to $d$ and $p>1$ so that $(X,(d')^p)$ is a metric space.
\item[(ii)] There exists a metric $d''$ on $X$ bi-Lipschitz equivalent to $d$ and $0\leq\alpha<1$ so that $(X,d'')$ satisfies the $\SRA(\alpha)$ condition.
\end{itemize}
\end{theorem}

Recall that two metrics $d$ and $d'$ on a set $X$ are said to be {\em bi-Lipschitz equivalent} if there exists a constant $L>0$ so that
\[
\frac{1}{L} d(x,y) \leq d'(x,y) \leq L d(x,y) \qquad \forall \, x, y\in X.
\]

For a given choice of $\alpha \in [0,1)$, a metric space $(X,d)$ is said to be {\it $\SRA(\alpha)$ free} if there exists $N \in \N$ so that any $\SRA(\alpha)$ subset of $X$ has cardinality at most $N$, while $(X,d)$ is said to be {\it $\SRA(\alpha)$ full} if it contains an infinite $\SRA(\alpha)$ subset. A result in geometric combinatorics due to Erd\"os and F\"uredi (Proposition \ref{ErdosFuredi}) implies that $\R^n$ is $\SRA(\alpha)$ free for each $\alpha<1$. Other examples of $\SRA$ free spaces include finite-dimensional Alexandrov spaces of non-negative curvature and Cayley graphs of virtually abelian groups. Moreover, all $\SRA(\alpha)$ free spaces are doubling, \cite[Theorem 6]{LOZ}. On the other hand, spaces containing large snowflaked subsets are $\SRA$ full in view of Theorem \ref{main-prop}. For example, the sub-Riemannian Heisenberg group $\Heis^1$ is $\SRA(\tfrac12)$ full, since the $t$-axis in $\Heis^1$ is isometric to a snowflaked copy of the real line. Other examples of $\SRA$ full spaces include the standard Laakso graphs (see \cite{Laa} and Example \ref{example:laakso}). 
An interesting class of $\SRA(0)$ full examples is provided by a family of metric trees $T_t$ whose construction depends on a sequence of real numbers $t$ (Example \ref{example:metrictree}). How one chooses the sequence can influence the doubling property, the bi-Lipschitz embeddability of the space, and the possibility of finding a bounded unrectifiable self-contracting curve (Proposition \ref{prop:sequence-embeddings}).
Infinite-dimensional spaces also provide examples of $\SRA$ full spaces; for instance, Hilbert space is $\SRA(\tfrac12)$ full (Example \ref{example:hilbert}). 

Earlier, we noted that the snowflaked space $(X,d^\eps)$ satisfies the $\SRA(2^\alpha-1)$ condition, and that the value $2^\alpha-1$ was best possible for such a conclusion without further restrictions on $X$. However, in contrast with this fact we prove the following result on the existence of large $\SRA(\eps)$ subsets in snowflaked metric spaces for arbitrary positive choices of $\eps$.

\begin{theorem}\label{thm:main-cor-2}
Let $(X,d)$ be any metric space containing a nontrivial geodesic curve, and let $0<\alpha<1$. Then $(X,d^\alpha)$ is $\SRA(\eps)$ full for each $0<\eps<\alpha$.
\end{theorem}

In fact, we can prove an even stronger conclusion: any such metric space $(X,d^\alpha)$ contains not only an infinite subset satisfying the $\SRA(\eps)$ condition, but contains Cantor-type subsets of positive Hausdorff dimension satisfying that condition (Proposition \ref{prop:cantor-set-subsets}).

In the second part of the paper, we connect the theory of $\SRA$ free metric spaces to the rectifiability of roughly self-contracting curves, a class of curves introduced in \cite{DaDuDe}. A curve $\gamma:I\to(X,d)$ is called a {\em rough $\lambda$-self-contracting curve} if for every $t_{1}\leq t_{2}\leq t_{3}$ in $I$ we have
\[ 
d(\gamma(t_{2}),\gamma(t_{3}))\leq d(\gamma(t_{1}),\gamma(t_{3}))+\lambda
d(\gamma(t_{1}),\gamma(t_{2})).
\]
The class of roughly self-contracting curves is a natural generalization of the class of {\em self-contracted curves} ($\lambda=0$) introduced in \cite{DLS}. Due to the role of self-contraction in the theory of gradient flows for convex potentials, it is of interest to know when such 
curves are rectifiable. In the Euclidean setting, solutions to gradient flow systems governed by a convex or quasi-convex function are steepest descent curves that satisfy the self-contractedness property. The rectifiability of such curves is linked to the convergence of various central optimization algorithms in convex analysis or graph theory (see \cite{DLS}, \cite{DDDL}, \cite {DaDuDe}, and the references therein). It is worth highlighting that the theory of gradient flows has also been studied in a purely metric context, through a variational inequality based on a metric characterization of the curves of maximum slope (see the seminal book by Ambrosio, Gigli, and Savar\'e \cite{AmGiSa}). In certain metric spaces, such as ${\mbox RCD}(K,\infty)$ spaces, gradient curves are also known to be self-contracting (\cite[Proposition 30]{LOZ}).

Rectifiability of self-contracting curves is known to hold in Euclidean space \cite{DLS}, \cite{DDDL}, \cite{MP}, 
in finite dimensional normed spaces \cite{ST}, \cite{Lem}, on Riemannian manifolds \cite{DDDR} and in a certain class of $
\mbox{CAT}(0)$ spaces \cite{Oh}. Zolotov \cite{Z} establishes the same conclusion for self-contracting curves in $\SRA(\alpha)$ free 
metric spaces for a suitable range of $\alpha$ (see also \cite{LOZ}). Rectifiability of rough $\lambda$-self-contracting curves in $
\R^n$ for $\lambda<\tfrac1n$ was established in \cite{DaDuDe}, but the topic has not been considered in any other settings to date. 
The main result of the second part of this paper generalizes the theorem of Zolotov \cite{Z} and the results in \cite{DaDuDe} to the setting of roughly self-contracting curves.

\begin{theorem}\label{thm:rectifiability}
Let $\tfrac12<\alpha<1$ and let $(X,d)$ be an $\SRA(\alpha)$ free metric space. Then there exists $\lambda_0 = \lambda_0(\alpha,X)>0$ so that any bounded, rough $\lambda$-self-contracting curve in $(X,d)$, with $\lambda \in [-1,\lambda_0)$ is rectifiable.
\end{theorem}

We end this introduction with an outline of the paper. In section \ref{sec:sra-section} we give basic facts about metric spaces satisfying the $\SRA(\alpha)$ condition. In particular, in subsection \ref{subsec:sra-and-snowflaking} we discuss the connection between the $\SRA(\alpha)$ condition and snowflaking, and we prove Theorem \ref{main-prop}. The Euclidean bi-Lipschitz embeddability of metric spaces with small rough angles follows from Theorem \ref{main-prop} and the Assouad Embedding Theorem. We comment briefly on this topic in subsection \ref{subsec:sra-embeddability}. Section \ref{section:subsets} concerns subsets of metric spaces, especially snowflake metric spaces, satisfying the $\SRA$ condition. Here we give the proof of Theorem \ref{thm:main-cor-2}. We also present a number of examples of $\SRA$ free and $\SRA$ full metric spaces. 

The second part of the paper consists of section \ref{sec:self-contracting}, where we study the rectifiability of roughly self-contracting curves in $\SRA(\alpha)$ free metric spaces. A substantial portion of this section is devoted to the proof of Theorem \ref{thm:rectifiability}, which closely follows the innovative argument put forward by Zolotov \cite{Z} in the self-contracting category. 

We conclude the paper with a list of open questions (section \ref{sec:questions}) and two appendices. In Appendix A (section \ref{appendix-a}) , we provide an example, due to Hyt\"onen, of a doubling metric space which is neither $\SRA(\alpha)$-free nor $\SRA(\alpha)$-full. In Appendix B (section \ref{appendix-b}), we present a partial result in support of one of the open questions mentioned in section \ref{sec:questions}.

\medskip

\noindent {\bf Acknowledgements:} We would like to thank Zolt\'an Balogh and Efstathios Chrontsios Garitsis for helpful discussions on the subject of this paper. We are also grateful to Tuomas Hyt\"onen for feedback, and for permission to include his example of a doubling metric space which is neither $\SRA$ free nor $\SRA$ full.

\section{Metric spaces with small rough angles}\label{sec:sra-section}

\subsection{Definitions}

Le Donne, Rajala, and Walsberg proved in \cite{LRW} that a snowflake of a metric space isometrically embeds into a finite-dimensional normed space if and only if the space is finite. The main idea behind the proof was to observe that, in a space with a snowflake metric, the angles formed by any three points must be {\em roughly small}. On the other hand, by a result of Erd\"os and F\"uredi \cite{EF}, in any Euclidean subset with sufficiently many points, some triple of points must form a {\em large} angle. Based on \cite{LRW}, Zolotov in \cite{Z} defined a class of metric spaces that satisfy a stronger form of the triangle inequality, while also capturing the idea that all angles within the space are {\em roughly} small.

\begin{definition}[Small rough angles condition]\label{def:SRA}
Let $0 \le \alpha \le 1$. A metric space $(X,d)$ satisfies the {\em $\SRA(\alpha)$ condition} ({\em small rough angles with parameter $\alpha$)} if
\begin{equation}\label{SRA}
d(x,y) \le \max \{ d(x,z) + \alpha d(z,y) , \alpha d(x,z) + d(z,y) \} \qquad \mbox{ for all $x,y,z \in X$.}
\end{equation}
\end{definition}

Condition \eqref{SRA} weakens as $\alpha$ increases, so every space satisfying the $\SRA(\beta)$ condition also satisfies the $\SRA(\alpha)$ condition for each $\alpha>\beta$. Every metric space satisfies the $\SRA(1)$ condition, while the $\SRA(0)$ condition characterizes the class of {\em ultrametric spaces}.

\begin{remark}\label{rem:SRA-alpha-negative}
In principle, one might also consider \eqref{SRA} for values $\alpha \in [-1,0)$.\footnote{In section \ref{sec:self-contracting} we will consider an analogous notion  ({\em rough $\lambda$-self-contractivity}) for ordered sets of points and for curves, in which a condition similar to \eqref{SRA} is imposed but only for {\bf ordered triples} of points. The notion of rough $\lambda$-self-contracting ordered set or curve makes sense and is nontrivial for all $-1\le\lambda\le 1$, with the rough $(-1)$-self-contracting condition characterizing collinear sets or geodesic curves. By way of contrast, the condition that a metric space satisfy the $\SRA(\alpha)$ condition imposes a restriction on all possible permutations of any triple of points.} However, it is easy to see that no space with more than two elements can satisfy the $\SRA(\alpha)$ condition for any $\alpha<0$. Assume that $(X,d)$ satisfies the $\SRA(\alpha)$ condition with $\alpha<0$, and suppose that $X$ contains at least three distinct elements $x,y,z$. Without loss of generality, assume that $d(x,y) \ge \max\{d(x,z),d(y,z)\}$. Interchanging the roles of $y$ and $z$ if necessary, we may assume from \eqref{SRA} that $d(x,y) \le d(x,z)+\alpha d(y,z)$. Then $d(x,z) \le d(x,z)+\alpha d(y,z)$ and hence $d(y,z)=0$. This contradicts the assumption that $x$, $y$, and $z$ are distinct.
\end{remark}

\begin{remark}\label{rem:SRA-remark}
In general, it can be difficult to find sets satisfying an $\SRA(\alpha)$ condition. For example, no set of three points on the real line satisfies the $\SRA(\alpha)$ condition for any $\alpha<1$. To see this, assume $x,z\in\R$ and assume without loss of generality that $d(x,z)=1$. Let us determine where a third point $y$ satisfying an $\SRA(\alpha)$ condition should be located. Without loss of generality, we can assume that $x< y< z$. If $d(x,y)=\delta$ for some $\delta<1$ then $d(y,z)=1-\delta$. In order for \eqref{SRA} to be satisfied, we should have
$$
1\leq\max\{\delta+\alpha(1-\delta),(1-\delta)+\alpha\delta\}.
$$
Hence, either $1\leq \delta+\alpha(1-\delta)$ (implying $\delta\geq 1$) or $1\leq(1-\delta)+\alpha\delta$ (implying $\delta\leq0$), and both cases lead to a contradiction.

Although we presented this example in the real line, the proof shows that the same conclusion holds true in any metric space that contains three collinear points.
\end{remark}

\begin{remark}\label{rem:SRA-angles}
Let us motivate why Definition \ref{def:SRA} inherently reflects the idea that all angles are {\em roughly} small, depending on the parameter $\alpha$. Any set of three points $x,y,z$ in a metric space $(X,d)$ can be isometrically embedded into $\R^2$. Let us consider a {\em comparison triangle} $\triangle(\bar{x},\bar{y},\bar{z}) \subset \R^2$, with angles $\widehat{xyz}, \widehat{yzx}$ and $\widehat{zxy}$. Assume that $d(x,z)\leq d(z,y)$. Then the validity of \eqref{SRA} for some given $\alpha\le1$ implies that $d(x,y) \le \alpha d(x,z) + d(z,y)$, whence
\[
d^2(x,y) \le \alpha^2 d^2(x,z) + d^2(z,y)+2\alpha d(x,z) d(z,y).
\]
On the other hand, the law of cosines and the previous inequality imply that
\[
\cos \widehat{yzx}=\frac{d^2(z,y)+d^2(x,z)-d(x,y)^2}{2d(z,y)d(x,z)}\geq -\alpha,
\]
and so the angle $\widehat{yzx}$ is less or equal than $\pi-\arccos(\alpha)$. The same argument applies to the other two angles. Hence any angle formed by any three points in $(X,d)$ is less than or equal to $\pi-\arccos(\alpha)$. In particular, if $(X,d)$ satisfies the $\SRA(0)$ condition then every angle is at most $\tfrac\pi2$. 
\end{remark}

\begin{remark}\label{rem:SRA-angles2}
We emphasize that the fact that all angles are less than or equal to $\pi-\arccos(\alpha)$ does not imply that the space satisfies the $\SRA(\alpha)$ condition. For an example, let $X:=\{x,y,z\}$ be a $3$-point subset of $\R^2$ where $x=(1/2,0),y=(-1/2,0)$ and $z=(0,1/2)$. In this case, $d(x,y)=1$, $d(x,z)=d(y,z)=\sqrt{2}/2$, and all angles are less than or equal to $\pi/2$. However, $X$ is not an ultrametric space. Recall that, in any ultrametric space, every triangle is an acute isosceles triangle, that is, a triangle with two equal sides with the third side of length less than or equal to that of the other two.
\end{remark}

\begin{remark}
Furthermore, a metric space satisfying an $\SRA(\alpha)$ condition for some $0<\alpha<1$ need not have the property that all triangles are acute. Choose $\delta>0$ such that $\sqrt{2(1+\sin\delta)}>\sqrt{2}$. Let $\alpha=\sqrt{2(1+\sin\delta)}-1$ and let $X=\{x,y,z\}$ be a $3$-point metric space such that $d(x,y)=1+\alpha$ and $d(x,z)=d(y,z)=1$. This space trivially satisfies the $\SRA(\alpha)$ condition. However, by the law of cosines, the euclidean angle $\widehat{xzy}$ in the comparison triangle satisfies 
\[
\cos(\widehat{xzy})=\frac{2-(1+\alpha)^2}{2}<0.
\]
\end{remark}

We recall the result of Erd\"os--F\"uredi \cite[Theorem 4.3]{EF} alluded to above: in $\R^n$, sufficiently large subsets necessarily form at least one {\em large angle}. In \cite[Theorem 1.1]{LRW}, this result is used in the context of finite-dimensional normed spaces with the help of the John Ellipsoid Theorem.

\begin{proposition}[Erd\"os--F\"uredi]\label{ErdosFuredi}
For any $n \in \mathbb{N}$ and $0<\beta<\pi$ there exists $K \in \mathbb{N}$ so that if $S \subseteq \mathbb{R}^n$ has cardinality at least $K$, then there are distinct $x, y, z \in S$ such that $\beta \leq \widehat{xyz} \leq \pi$.
\end{proposition}

According to \cite[Theorem 4.3]{EF}, the optimal choice of $K$ satisfies
\begin{equation}\label{EFquantitative}
2^{\left( \tfrac{\pi}{\pi-\beta} \right)^{n-1}} \le K \le 2^{\left( \tfrac{4\pi}{\pi-\beta} \right)^{n-1}}.
\end{equation}
The proof in \cite{EF} is combinatorial; a purely geometric proof for Proposition \ref{ErdosFuredi} (which yields a non-sharp value for $K$) was given by K\"aenm\"aki and Suomala in \cite{KS}.

\begin{corollary}\label{EFcor}
For each $n \in \N$ and $\alpha<1$, there exists $K \in \N$ so that if $S\subset \R^n$ satisfies the $\SRA(\alpha)$ condition, then $S$ contains at most $K$ elements.
\end{corollary}

\begin{remark}\label{EFremark}
The cardinality of any set of points in $\R^n$, $n \ge 2$, satisfying the $\SRA(\alpha)$ ($0 \le \alpha \le 1$) condition provides a lower bound for the maximal number of points in $\R^n$ such that all angles determined by any triple of points are less or equal to a certain angle $\beta=\pi-\arccos(\alpha)$ ($\pi/2\leq \beta \le \pi$). For example, a regular simplex in $\R^n$ (comprising $n+1$ points) trivially satisfies the $\SRA(\alpha)$ condition for any $\alpha<1$. Moreover, the vertices of a regular simplex together with the center of its circumscribed spFhere form a set of $n+2$ points that satisfies the $\SRA(\sqrt{2(n+1)/n}-1)$ condition. 

In view of \eqref{EFquantitative} and Remark \ref{rem:SRA-angles}, the maximal cardinality of a set of points in $\R^n$ satisfying the $\SRA(\alpha)$ condition is at least
$$
A(n)^{(\arccos\alpha)^{1-n}}
$$
for some constant $A(n)>1$.
\end{remark}

We will return to this circle of ideas in Section \ref{section:subsets}, where we introduce the notions of {\em $\SRA(\alpha)$ free} and {\em $\SRA(\alpha)$ full} metric spaces. Phrased in that language, Corollary \ref{EFcor} says that $\R^n$ is an $\SRA(\alpha)$ free metric space for each $n \in \N$ and $\alpha<1$. Before we take up that discussion, we consider how the snowflaking transformation of a metric affects the validity of the $\SRA(\alpha)$ condition.

\subsection{Small rough angles and the snowflaking transformation on metric spaces}\label{subsec:sra-and-snowflaking}

In \cite[Theorem 1.1]{LRW}, the authors establish that for a given metric space $(X,d)$ and $0<\alpha<1$, the snowflaked metric space $(X,d^\alpha)$ satisfies the $\SRA(\beta)$ condition with $\beta = \alpha$. The following result improves this observation and gives the sharp value of $\beta$.

\begin{lemma}\label{lem:bestSRA}
Let $(X,d)$ be a metric space and $0<\alpha<1$. Then $(X,d^\alpha)$ satisfies the $\SRA(2^\alpha-1)$ condition.
\end{lemma}

\begin{proof}
Let $0<\alpha<1$ and $x,y,z\in X$. Let us denote $A=d(x,y)$, $B=d(y,z)$ and $C=d(x,z)$. Without loss of generality we can assume $A\leq B\leq C$. Since $C\leq A+B$ it is enough to find the smallest possible $0<\beta<\alpha$ such that 
\[
(A+B)^{\alpha}\leq\max\{A^{\alpha}+\beta B^{\alpha},\beta A^{\alpha}+ B^{\alpha}\}.
\]
Because $A\leq B$, the maximum is attained in the second term. Dividing the expression by $B^{\alpha}$ and setting $t=A/B$, the problem reduces to finding $\beta<\alpha$ such that the function
\[
f(t):=1+\beta t^{\alpha}-(1+t)^{\alpha}
\]
satisfies $f(t)\geq 0$ for $0< t<1$. If we impose $f(1)=0$ we obtain $\beta=2^\alpha-1$. Because $f(0)=0$, $f$ is continuous, and has a single maximum in the interval $(0,1)$ it is clear that for $\beta=2^\alpha-1$, $f(t)\geq 0$ for $0< t<1$ and the result follows.
\end{proof}

In particular, since $2^\alpha-1<\alpha$ for each $0<\alpha<1$, we recover the conclusion of \cite{LRW}. The subsequent example demonstrates that Lemma \ref{lem:bestSRA} is sharp.

\begin{example}\label{ex:arithmetic}
Let $X=\{x,y,z\}$ be an arithmetic sequence in $\R$, with $C=|x-z| $ and $A=|x-y|=|y-z|=C/2$. Then $C = A+A$ and for any $0<\alpha<1$, $C^\alpha =  (2^\alpha-1) A^\alpha+A^\alpha$.
\end{example}

A natural question is whether the converse of Lemma \ref{lem:bestSRA} is true; specifically, if the $\SRA(\beta)$ condition implies that the metric coincides with a fractional power of another metric. The following example shows that this is not the case.

\begin{example}\label{ex:no-converse}
Let $Z = B \times \N$ where $B = \{z \in \C:|z| \le 1\}$, and equip $Z$ with the following metric:
$$
d((z,m),(w,p)) = \begin{cases} |z-w| & \mbox{if $m=p$,} \\
2 & \mbox{if $m \ne p$.} \end{cases}
$$
Fix $0<\beta<1$ and a sequence $(\delta_m)$ with $\delta_m \le 1$ and $\delta_m \searrow 0$. Let $x_m,y_m,z_m$, $m \in \N$, be points in $B$ with the property that for fixed $m$, the points $x_m$, $y_m$, $z_m$ form the vertices of a triangle with sides of Euclidean length $1$, $\delta_m$, and $1+\beta\delta_m$. Let $X \subset Z$ be the subset consisting of all points of the form $(p,m)$, where $p \in \{x_m,y_m,z_m\}$ and $m \in \N$. 

We first verify that $(X,d)$ satisfies the $\SRA(\beta)$ condition. Let $(p,k)$, $(q,\ell)$, and $(r,m)$ be distinct elements of $X$. If $k=\ell=m$ then $\{p,q,r\} = \{x_m,y_m,z_m\}$ and the $\SRA(\beta)$ condition is clearly satisfied for any permutation of the three points. On the other hand, if $k$, $\ell$, and $m$ are not all equal then without loss of generality assume that $k \ne \ell$ and $k \ne m$. In this case the triangle formed by the points $(p,k)$, $(q,\ell)$, and $(r,m)$ is an isoceles triangle with two sides of length $2$ and a third side of length less than or equal to $2$. Such triangle is then an acute isosceles triangle in the sense of Remark \ref{rem:SRA-angles2}, and hence this triple of points satisfies the $\SRA(0)$ condition and consequently also satisfies the $\SRA(\beta)$ condition.

Finally, we show that $(X,d^q)$ fails to be a metric space for any $q>1$. Suppose that $d^q$ is a metric on $X$ for some $q>1$. Then the triangle inequality
\begin{equation}\label{eq:d-q-failure}
(1+\beta\delta_m)^q \le 1 + \delta_m^q \qquad \forall \, m \in \N.
\end{equation}
Using the inequality $1+qx \le (1+x)^q$ valid for $x>0$ and $q>1$, we obtain $1 + q \beta \delta_m \le 1 + \delta_m^q$ and hence
$$
0< q \beta \le \delta_m^{q-1}.
$$
Since $q>1$ and $\delta_m \to 0$ we get a contradiction.
\end{example}

\begin{remark}
The metric space $(Z,d)$ defined in the previous example is separable, and hence embeds isometrically in $\ell^\infty$ by the Fr\'echet embedding theorem. It is natural to ask if any examples illustrating this conclusion can be provided which lie in a nicer Banach space, e.g., $\ell^p$ for some $1<p<\infty$ or even $\ell^2$. In Appendix B (section \ref{appendix-a}), we give an example of a set $X \subset \ell^2$ of this type, but only for a restricted range of values of $\alpha$. The construction is essentially the same as in Example \ref{ex:no-converse}, but rather than abstractly defining the distance between elements of distinct triangles to be equal to $2$, the triangles are arranged in a sequence of orthogonal two-dimensional subspaces of $\ell^2$, lying near the vertices of an infinite-dimensional equilateral simplex.

See section \ref{sec:questions} for further questions and remarks.
\end{remark}

Next, we show that if $X$ is a finite metric space, then the converse of Lemma \ref{lem:bestSRA} is true.

\begin{proposition}
If a finite set satisfies the $\SRA(\alpha)$ condition with respect to a metric $d$, then the metric coincides with a fractional power of another metric. 
\end{proposition}

\begin{proof}
Let $(X=\{x_1,x_2,...,x_n\},d)$ be a metric space. We can construct $N=\frac{n!}{3! (n-3)!}$ comparison triangles $\Delta_t$ ($1\leq t\leq N$) with sides of Euclidean lengths $d(x_i,x_j)$, $d(x_j,x_k)$ and $d(x_i,x_k)$ for every $i\neq j\neq k$, where $i,j,k\in\{1,2,\cdots, n\}$. Fix a triangle $\Delta_t$ with vertices $x_i,x_j$ and $x_k$, and denote $c_t=d(x_i,x_j),b_t=d(x_j,x_k)$, and $a_t=d(x_i,x_k)$. Without loss of generality, we can assume that $0<c_t\leq b_t\leq a_t$. 
Since $(X,d)$ satisfies the $\SRA(\alpha)$ condition, it follows that
\[
a_t\leq\max\{b_t+\alpha c_t,c_t+\alpha b_t\}.
\]
Let us first consider the case where $a_t \le b_t+\alpha c_t$. If
\begin{equation}\label{eq:dptriangle}
(b_t+\alpha c_t)^p\leq b_t^p + c_t^p
\end{equation}
for some choice of $p>1$, then the triple of points $a_t,b_t,c_t$ satisfies the triangle inequality for the $d^p$ metric. For $x\ge 1$ and $0< \alpha <1$, define a function $f_{x,\alpha}:[1,\infty) \to \R$ by
\[
f_{x,\alpha}(p)=1+x^p-(x+\alpha)^p.
\]
To ensure inequality \eqref{eq:dptriangle} holds, we must find $p_t = p_t(c_t,b_t,\alpha)>1$ so that
\begin{equation}\label{eq:dptriangle2}
f_{b_t/c_t,\alpha}(p_t) \ge 0.
\end{equation}
If $c_t = b_t$ then we may select any $p = p(\alpha)$ satisfying $1<p\le \tfrac{\log 2}{\log(1+\alpha)}$ and observe that $f_{1,\alpha}(p) = 2-(1+\alpha)^p \ge 0$. If $c_t < b_t$ then we consider $f_{x,\alpha}$ for $x = b_t/c_t > 1$. Observe that $f_{x,\alpha}(1) = 1-\alpha>0$ for all $x$ and $\alpha<1$, and that $f_{x,\alpha}(p) \to - \infty$ as $p \to \infty$ provided $\alpha>0$. Furthermore, $f_{x,\alpha}'(p) = x^p \log x - (x+\alpha)^p \log(x+\alpha) < 0$ since the function $y \mapsto y^p \log y$ is strictly increasing for $y \in (1,\infty)$. Hence $f_{x,\alpha}(p)$ is strictly decreasing for $p \in (1,\infty)$ and there exists a unique $p = p(x,\alpha) > 1$ so that $f_{x,\alpha}(p) = 0$. We set $p_t = p(b_t/c_t,\alpha)$ and conclude that \eqref{eq:dptriangle2} is satisfied.

Next, we consider the case where $a_t \le c_t+\alpha b_t$. In this case, if
\begin{equation}\label{eq:dptriangle3}
(c_t+\alpha b_t)^p\leq b_t^p + c_t^p
\end{equation}
for some $p>1$, then the triple of points $a_t,b_t,c_t$ satisfies the triangle inequality for the $d^p$ metric. Again, for $x \ge 1$ and $0<\alpha<1$ we define a function $g_{x,\alpha}:[1,\infty) \to \R$ by
\[
g_{x,\alpha}(p)=1+x^p-(1+\alpha x)^p.
\]
To ensure inequality \eqref{eq:dptriangle3} holds, we must find $p_t = p_t(c_t,b_t,\alpha)>1$ so that
\begin{equation}\label{eq:dptriangle4}
g_{b_t/c_t,\alpha}(p_t) \ge 0.
\end{equation}
If $c_t = b_t$ then we may select any $p = p(\alpha)$ satisfying $1<p\le \tfrac{\log 2}{\log(1+\alpha)}$ as before. If $c_t<b_t$ then we consider $g_{x,\alpha}$ for $x = b_t/c_t>1$. Observe that $g_{x,\alpha}(1) = (1-\alpha)x > 0$ for all $x$ and $\alpha<1$. If $x>1 + \alpha x$ then $g_{x,\alpha}(p) \to + \infty$ as $p \to \infty$ and $g_{x,\alpha}(p)$ is strictly increasing for $p\in(1,\infty)$; in this case any choice of $p>1$ is allowable. If $x<1+\alpha x$ then $g_{x,\alpha}(p) \to - \infty$ as $p \to \infty$ and $g_{x,\alpha}(p)$ is strictly decreasing for $p\in(1,\infty)$. As in the previous case there exists a unique $p = p(x,\alpha)>1$ so that $g_{x,\alpha}(p) = 0$. We set $p_t = p(b_t/c_t,\alpha)$ and conclude that \eqref{eq:dptriangle4} is satisfied. Finally, if $x = 1+\alpha x$ then $g_{x,\alpha}(p) = 1$ for all $p$ and any choice of $p>1$ is allowable. In all cases, we have found a suitable choice for $p_t>1$.

To conclude, observe that $d^p$ is a metric on $X$ where $p>1$ is defined by $p:=\min\{p_t: 1\leq t\leq N\}$. To see this, it suffices to observe that an $L^p$ metric is also an $L^q$ metric for $q<p$. Indeed, if
\[
d(x,y)^p\leq d(x,z)^p+d(y,z)^p,
\]
then, because $q/p<1$,
\[
d(x,y)^q=(d(x,y)^p)^{q/p}\leq (d(x,z)^p+d(y,z)^p)^{q/p}\leq d(x,z)^q+d(y,z)^q,
\]
as wanted.
\end{proof}

In the rest of this section, we prove Theorem \ref{main-prop}, which asserts the equivalence of the $\SRA$ and snowflaking conditions for arbitrary metric spaces up to bi-Lipschitz distortion of the metric. Let us first clarify what we mean by a snowflake metric, following the definition in \cite{TW}.

\begin{definition}[Snowflake metric space]
Let $1\le p < \infty$. A metric $d$ on a space $X$ is said to be an {\it $L^p$-metric} if
$$
d(x,y)^p \le d(x,z)^p + d(z,y)^p \quad \mbox{ for all $x, y,z \in X$.}
$$
A metric $d$ is said to be an {\it $L^\infty$-metric} (alternatively, an {\it ultrametric}), if
$$
d(x,y) \le \max\{ d(x,z) , d(z,y) \} \qquad \mbox{ for all $x,y,z \in X$.}
$$
For $p \in [1,\infty]$, a metric space $(X,d)$ is said to be a {\it $p$-snowflake} if there exists an $L^p$ metric $d'$ on $X$ so that $d$ and $d'$ are bi-Lipschitz equivalent. If $(X,d)$ is a $p$-snowflake for some $1<p\le\infty$, we say that $(X,d)$ is a {\it snowflake}.
\end{definition}

\begin{remark}
For $p<\infty$, a metric $d$ on $X$ is an $L^p$ metric if and only if there exists a metric $d'$ on $X$ so that $d = (d')^\eps$, where $\eps = 1/p$. 
\end{remark}

\begin{example}
Let $\Heis^1$ be the first Heisenberg group equipped with the Carnot--Carath\'eodory metric $d_{cc}$. Let
\begin{equation}\label{eq:X}
X = \{ (0,t) : 0 \le t \le T\} \subset \Heis^1
\end{equation}
be a bounded line segment on the $t$-axis in $\Heis^1$. Recall that there exists some universal constant $c>0$ so that for any $0\le s < t \le T$,
\begin{equation}\label{eq:d-cc-s-t}
d_{cc}((0,s),(0,t)) = c |s-t|^{1/2}.
\end{equation}
Thus the restriction of $d_{cc}$ to $X$ is an $L^2$-metric.
\end{example}

For the convenience of the reader, we restate Theorem \ref{main-prop}. 

\begin{theorem}\label{main-prop-2}
For a metric space $(X,d)$, the following conditions are quantitatively equivalent:
\begin{itemize}
\item[(i)] $X$ is a $p$-snowflake for some $1<p\le\infty$.
\item[(ii)] There exists a metric $d'$ on $X$ so that $d'$ is bi-Lipschitz equivalent to $d$ and $(X,d')$ satisfies the $\SRA(\alpha)$ condition for some $0\leq \alpha<1$.
\end{itemize}
\end{theorem}

The implication (i) $\Rightarrow$ (ii) follows from Lemma \ref{lem:bestSRA}. More precisely, that lemma implies that if $d$ is an $L^p$ metric on $X$ for some $p>1$, then $(X,d)$ satisfies the $\SRA(2^{1/p}-1)$ condition. We focus on the reverse implication (ii) $\Rightarrow$ (i). To this end, we consider a third condition on metric spaces of a similar nature to the $L^p$-metric and $\SRA(\alpha)$ conditions.

\begin{definition}[Uniformly non-convex metric spaces]\label{def:UNC}
Let $0<\delta<\tfrac12$. A metric space $(X,d)$ is {\it $\delta$-uniformly non-convex} ($\delta$-UNC) if for every $x,y \in X$ there exists $\lambda \in (\delta,1-\delta)$ so that the set
$$
B(x,(\lambda+\delta)d(x,y)) \cap B(y,(1-\lambda+\delta)d(x,y))
$$
is empty. Here $B(x,r)$ denotes the closed ball in $X$ with center $x$ and radius $r$.
\end{definition}

The uniform non-convexity condition was also introduced in \cite{TW}. Intuitively, a space $(X,d)$ is UNC if for every pair of points $x,y \in X$, there is a relatively large gap along the straight line path between $x$ and $y$. The precise condition in the definition is formulated to account for the fact that in a general metric space, such straight line paths may not exist. If $(X,d)$ is a normed vector space, then the UNC condition can be reformulated as a {\it uniformly linearly non-convex} (ULNC) condition using the segment $[x,y]$ for $x,y \in X$; see \cite[Definition 3.4]{TW} for details.

We make use of the following result which can be found in \cite[Theorem 1.5]{TW}, see also the subsequent comments.

\begin{proposition}[Tyson--Wu]\label{prop:TW}
For each $0<\delta<\tfrac12$ there exists $q=q(\delta)>0$ so that if $(X,d)$ be a $\delta$-UNC metric space, then $(X,d)$ is a $q$-snowflake. 
\end{proposition}

The proof in \cite{TW} shows that we may choose
$$
q = \frac{\log 2}{\log 2 - \log(1+4\delta^2)}.
$$

To complete the proof of the implication (ii) $\Rightarrow$ (i) in Theorem \ref{main-prop-2} it suffices to prove the following lemma.

\begin{lemma}\label{lemma:auxUNC}
For each $0<\alpha<1$ there exists $\delta = \delta(\alpha) \in (0,\tfrac12)$ so that if $(X,d)$ is an $\SRA(\alpha)$ metric space then $(X,d)$ is $\delta$-UNC.
\end{lemma}

\begin{proof}
For $0<\alpha<1$ we set
$$
\delta = \frac12 \cdot \frac{1-\alpha}{1+\alpha}.
$$
Assume that $(X,d)$ is an $\SRA(\alpha)$ metric space. Let $x,y \in X$. We set $\lambda = \tfrac12$ and let $z \in X$ be arbitrary. We must show that either
$d(x,z) > (\lambda+\delta) d(x,y)$ or $d(z,y) > (1-\lambda+\delta) d(x,y)$.

Set $A = d(x,z)$, $B=d(z,y)$, and $C=d(x,y)$. Without loss of generality we may assume that $A\neq B$. By assumption, $C \le \max \{ A + \alpha B, \alpha A + B \}$. We must show that
$$
C < \max \left\{ \frac{A}{\lambda+\delta} , \frac{B}{1-\lambda+\delta} \right\} = \frac1{\tfrac12+\delta} \max\{A,B\}.
$$
For the given choice of $\delta = \delta(\alpha)$, observe that
$$
\frac1{\tfrac12+\delta} = 1+\alpha.
$$
It thus suffices to prove that $\max \{ A + \alpha B, \alpha A + B \} < (1+\alpha) \max\{A,B\}$, which is obvious. 
\end{proof}

\begin{remark}
The above proof, together with Proposition \ref{prop:TW}, shows that every metric space bi-Lipschitz equivalent to a metric satisfying the $\SRA(\alpha)$ condition for some $0<\alpha<1$ is a $q$-snowflake with
\begin{equation}\label{eq:q}
q = q(\alpha) = \frac{\log 2}{\log(1+\tfrac{2\alpha}{1+\alpha^2})}.
\end{equation}
\end{remark}

\begin{remark}
The special case $\alpha=0$ in the spectrum of $\SRA(\alpha)$ conditions marks a qualitative change in the topology of the space. Recall that $(X,d)$ satisfies the $\SRA(0)$ condition if and only if $(X,d)$ is ultrametric, and every ultrametric space is totally disconnected. On the other hand, the $\SRA(\beta)$ condition for positive $\beta$ is satisfied by the snowflaked space $(X,d^\beta)$ for {\bf any} given metric space $(X,d)$, and the snowflaking transformation $(X,d) \mapsto (X,d^\beta)$ does not affect the topology of the space. It is thus easy to construct examples of connected metric spaces which satisfy the $\SRA(\beta)$ condition for any given positive $\beta$.

After allowing for a bi-Lipschitz change of metric, the following are equivalent:
\begin{itemize}
\item[$\bullet$] $(X,d)$ is bi-Lipschitz equivalent to an ultrametric space,
\item[$\bullet$] $(X,d)$ is bi-Lipschitz equivalent to a metric space satisfying the $\SRA(0)$ condition,
\item[$\bullet$] $(X,d)$ is an $\infty$-snowflake,
\item[$\bullet$] $(X,d)$ is uniformly disconnected (\cite{davidsemmes}).
\end{itemize}
\end{remark}

\subsection{Bi-Lipschitz embeddability of metric spaces with small rough angles}\label{subsec:sra-embeddability}

A classical theorem by Assouad \cite{A} states that every doubling snowflake metric space can be bi-Lipschitz embedded into some finite-dimensional Euclidean space. In view of the equivalence between the $\SRA$ condition (up to bi-Lipschitz equivalence) and the snowflake condition, we obtain a corresponding statement for metric spaces with small rough angles.

Recall that a metric space is called {\em (metrically) doubling} if there is a constant $C$ so that every ball of radius $r$ can be covered by at most $C$ balls of radius $r/2$.

\begin{proposition}\label{prop:SRAembeddability}
Let $(X,d)$ be a doubling metric space satisfying the $\SRA(\alpha)$ condition for some $\alpha<1$. Then $(X,d)$ embeds into some finite-dimensional Euclidean space $\R^N$ by a bi-Lipschitz map.
\end{proposition}

In fact, assuming that $(X,d)$ satisfies the $\SRA(\alpha)$ condition for some $\alpha<1$, then $(X,d)$ embeds bi-Lipschitzly into some finite-dimensional Euclidean space if and only if $(X,d)$ is doubling.

\begin{proof}
If $(X,d)$ satisfies the $\SRA(\alpha)$ condition for some $\alpha<1$, then $(X,d)$ is a $q$-snowflake for some $q>1$. The conclusion then follows from Assouad's embedding theorem.
\end{proof}

Recall that a metric space is doubling if and only if it has finite Assouad dimension. It is well known that the doubling property alone is not sufficient to ensure bi-Lipschitz embeddability into a finite-dimensional Euclidean space; the canonical counterexample in this regard is the sub-Riemannian Heisenberg group \cite{Semmes1996}. 

\begin{remark}
The question of determining the best possible exponent $N$ for a Euclidean target space in Assouad's theorem (in terms of the snowflaking parameter and the Assouad dimension of the source) has been studied extensively, see for instance \cite{NN}, \cite{DaSn}, \cite{Tao}, \cite{Ry}. We briefly comment on the corresponding problem for spaces satisfying a small rough angles condition, relating the minimal possible target dimension for these two problems via Theorem \ref{main-prop-2}.

More precisely, let us denote by $N_\snowflake(n,p)$, $1<p\le\infty$, the smallest positive integer $N$ so that every $L^p$ metric space $(X,d)$ with Assouad dimension strictly less than $n$ bi-Lipschitz embeds into $\R^N$. Similarly, denote by $N_\sra(n,\alpha)$ the smallest positive integer $N$ so that every metric space $(X,d)$ satisfying the $\SRA(\alpha)$ condition and with Assouad dimension strictly less than $n$ bi-Lipschitz embeds into $\R^N$.

It follows from a theorem of Luukkainen and Movahedi-Lankarani \cite[Proposition 3.3]{LM} that $N_\snowflake(n,\infty) = n$. Since the $L^\infty$ metric condition and the $\SRA(0)$ condition are equivalent, we also have $N_\sra(n,0) = n$.

Naor and Neiman \cite{NN} proved that $N_\snowflake(n,\alpha)$ is bounded above by a constant depending only on $n$ provided $\alpha \in [\tfrac12,1)$, see also David--Snipes \cite{DaSn} for an alternate proof of this fact.

\begin{proposition}\label{prop:N-snowflake-vs-N-sra}
For any $n \ge 1$ and $0\le\alpha<1$, we have
\begin{equation}\label{eq:N-snowflake-to-N-sra}
N_{\mbox \em snow}(n,\alpha) \le N_{\mbox {\em SRA}}\left( \left\lceil \frac{n}{\alpha} \right\rceil, 2^\alpha-1 \right)
\end{equation}
and
\begin{equation}\label{eq:N-sra-to-N-snowflake}
N_{\mbox {\em SRA}}(n,\alpha) \le N_\snowflake\left( \left\lceil \frac{n}{q(\alpha)} \right\rceil, \frac1{q(\alpha)} \right),
\end{equation}
where $q(\alpha)$ is defined in \eqref{eq:q}.
\end{proposition}

\begin{proof}
Let $(X,d)$ be a metric space with Assouad dimension strictly less than $n$ and let $0<\alpha<1$. Then the Assouad dimension of $(X,d^\alpha)$ is strictly less than $\tfrac{n}{\alpha}$ and, by Lemma \ref{lem:bestSRA}, $(X,d^\alpha)$ satisfies the $\SRA(2^\alpha-1)$ condition. The first inequality \eqref{eq:N-snowflake-to-N-sra} follows.

Conversely, let $(X,d)$ satisfy the $\SRA(\alpha)$ condition and have Assouad dimension strictly less than $n$. Then $(X,d)$ is a $q(\alpha)$-snowflake by Theorem \ref{main-prop}, i.e., $d$ is bi-Lipschitz equivalent to $(d')^{1/q(\alpha)}$ for some metric $d'$ on $X$. Moreover, the Assouad dimension of $(X,d')$ is strictly less than $n/q(\alpha)$ and the relevant snowflaking parameter is $1/q(\alpha)$. The second inequality \eqref{eq:N-sra-to-N-snowflake} follows.
\end{proof}
\end{remark}

We also point the reader to Remark \ref{rem:sra-and-embeddings} for more information about the $\SRA(\alpha)$ condition and bi-Lipschitz embeddability.

The following well-known result will be used in the proof of Proposition \ref{prop:sequence-embeddings} in the following section. This result has appeared several times in the literature, see e.g.\ \cite{lem:isosceles} and \cite[Theorem 6.7]{abbw:ultrametric}.

\begin{theorem}[Lemin, Aschbacher--Baldi--Baum--Wilson]\label{th:finite-ultrametric-embeddability}
Let $(X,d)$ be an ultrametric space of cardinality $N+1$. Then $X$ can be isometrically embedded into $\R^N$.
\end{theorem}

\section{$\SRA(\alpha)$ subsets of metric spaces}\label{section:subsets}

In this section we explore different settings in which a given metric space $(X,d)$ is guaranteed either to be free of sufficiently large subsets with the $\SRA(\alpha)$ condition or to contain arbitrarily large (or even infinite) subsets with that condition. These conditions turn out to be relevant for the question of whether or not all rough self-contracting curves are rectifiable; see Section \ref{sec:self-contracting}.

\subsection{$\SRA(\alpha)$ free and $\SRA(\alpha)$ full metric spaces}

The following definitions are taken from \cite{Z} and \cite{LOZ}. We denote by $\#S$ the cardinality of a finite set $S$.

\begin{definition}[$\SRA(\alpha)$ free and full spaces]\label{definition:SRAfree}
Let $0\leq\alpha<1$. 
\begin{enumerate}
\item A metric space $(X,d)$ is said to be {\em $\SRA(\alpha)$ free} if there exists $N \in \N$ such that for each $F\subset X$, if $F$ satisfies the $\SRA(\alpha)$ condition then $\#F\leq N$. 
\item A metric space $(X,d)$ is {\em $\SRA(\alpha)$ full} if there exists a subset $F\subset X$ which satisfies the $\SRA(\alpha)$ condition and has $\#F=\infty$.
\end{enumerate}
\end{definition}

It follows directly from these definitions that if $(X,d)$ is a metric space and $\alpha \le \beta$ then
\begin{itemize}
\item if $(X,d)$ is $\SRA(\beta)$ free, then $(X,d)$ is $\SRA(\alpha)$ free, and
\item if $(X,d)$ is $\SRA(\alpha)$ full, then $(X,d)$ is $\SRA(\beta)$ full.
\end{itemize}
We emphasize that the $\SRA(\alpha)$ full condition is stronger than the condition that the space not be $\SRA(\alpha)$ free. A space is not $\SRA(\alpha)$ free if it contains subsets of arbitrarily large cardinality which satisfy the $\SRA(\alpha)$ condition, while it is $\SRA(\alpha)$ full if it contains an infinite subset with that property. However, the $\SRA(\alpha)$ condition is usually not preserved under taking unions of sets. There exist metric spaces which are not $\SRA(\alpha)$ free but are also not $\SRA(\alpha)$ full; see Appendix A (section \ref{appendix-a}).

\begin{example}[Euclidean space]
$\R^n$ is $\SRA(\alpha)$ free for any $\alpha<1$, cf.\ Corollary \ref{EFcor}.
\end{example}

Other examples of $\SRA(\alpha)$ free spaces (see, e.g., \cite[Theorem 2]{LOZ}) include:
\begin{itemize}
\item finite-dimensional Alexandrov spaces of non-negative curvature,
\item finite-dimensional normed vector spaces,
\item complete Berward spaces of non-negative flag curvature,
\item Cayley graphs of virtually abelian groups.
\end{itemize}

By way of contrast, the next set of examples are $\SRA(\alpha)$ full for some $0\leq\alpha<1$.

\begin{example}[Heisenberg group]\label{example:heis}
The first Heisenberg group $\Heis^1$ equipped with the Carnot--Carath\'eodory metric $d_{cc}$ is $\SRA(\sqrt{2}-1)$ full. Indeed, recall that the restriction of $d_{cc}$ to the infinite set $X = \{(0,t):0 \le t \le T\}$ is an $L^2$-metric, so this follows from Lemma \ref{lem:bestSRA}. Note also that $(X,d_{cc})$ does not contain any ultrametric subsets with three or more points; in an ultrametric space, each comparison triangle is isosceles with the two equal sides being the longer ones, but in $X$ each comparison triangle is a right triangle. As a consequence of Proposition \ref{th:SRA-subsets-of-snowflake-spaces}, we will see that $\Heis^1$ equipped with the the metric $d_{cc}$, is in fact $\SRA(\alpha)$ full for any $0<\alpha<1$.
\end{example}

\begin{example}[Hilbert space]\label{example:hilbert}
The Hilbert space $\ell^2$ is $\SRA(\sqrt2-1)$ full. Denote by $(e_k)$ the canonical orthonormal basis for $\ell^2$, i.e., $e_k$ is $1$ in the $k$th position and $0$ in all other positions. Choose a sequence $(c_k)$ of positive numbers with $c_k \to 0$ and let $F = \{x_k:k \in \N\}$, where $x_k = c_k e_k$. Note that $d(x_k,x_\ell) = (c_k^2+c_\ell^2)^{1/2}$ for any $k,\ell \in \N$.

To see that $(F,d)$ satisfies the $\SRA(\alpha)$ condition with $\alpha = \sqrt2-1$, we need to show that
\begin{equation}\label{eq:hilbert1}
\sqrt{c_k^2+c_\ell^2} \le \max \left\{ \sqrt{c_k^2+c_m^2} + \alpha \sqrt{c_\ell^2 + c_m^2} , \alpha \sqrt{c_k^2 + c_m^2} + \sqrt{c_\ell^2 + c_m^2} \right\} \quad \mbox{for all $k,\ell,m \in \N$.}
\end{equation}
The validity of \eqref{eq:hilbert1} is equivalent to the validity of
\begin{equation}\label{eq:hilbert2}
\sqrt{c_k^2+c_\ell^2} \le \max \left\{ c_k + \alpha c_\ell , \alpha c_k + c_\ell \right\} \qquad \mbox{for all $k,\ell \in \N$;}
\end{equation}
one direction is obvious, and for the other direction, fix $k$ and $\ell$ and let $c_m \to 0$. Squaring both sides of \eqref{eq:hilbert2} and rearranging leads to the following equivalent formulation:
\begin{equation}\label{eq:hilbert3}
0 \le \max \left\{ 2\alpha c_k c_\ell - (1-\alpha^2) c_\ell^2, 2 \alpha c_k c_\ell - (1-\alpha^2) c_k^2 \right\} \qquad \mbox{for all $k,\ell \in \N$.}
\end{equation}
Since $\alpha = \sqrt2-1$ we have $2\alpha = 1-\alpha^2$ and so \eqref{eq:hilbert3} reads
$$
0 \le \max \{ c_\ell (c_k-c_\ell) , c_k (c_\ell-c_k) \} \qquad \mbox{for all $k,\ell \in \N$,}
$$
which is clearly satisfied.
\end{example}

\begin{example}[Metric trees]\label{example:metrictree}
This example is taken from \cite[Section 8.2]{LOZ}, where it is used to illustrate the fact that a space can contain arbitrarily large (even infinite) ultrametric subsets and still have the property that all bounded self-contracting curves are rectifiable. We discuss the relationship between the $\SRA(\alpha)$ free condition and rectifiability of rough self-contracting curves in the following section.

Let $t=\{t_i\}$ be a strictly decreasing sequence such that $\lim_{i\to\infty}t_i=0$.  For each $n \in \N$, let $C_i$ be the closed vertical line segment connecting the points $(t_i, 0)$ and $p_i=(t_i,t_i)$.  Additionally, let
$C_0$ be the horizontal segment connecting $(0,0)$ to $(t_i,0)$.   Now, define $T_t=\cup_{i=0}^\infty C_i$ and consider the intrinsic distance on $T_t$, denoted by $d_T$.
\begin{figure}[h]
\begin{center}
\includegraphics[width=5.5cm]{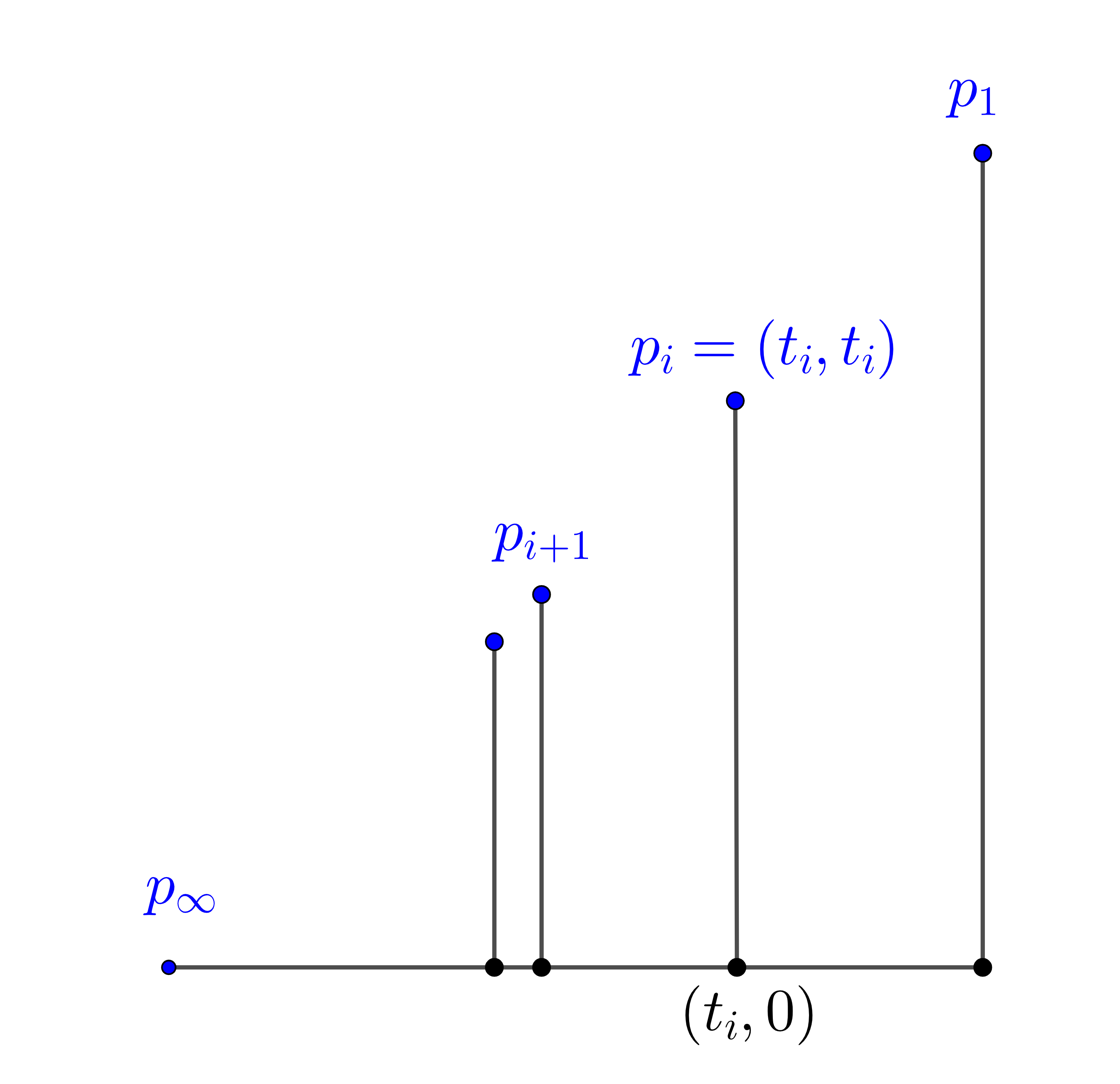}
\caption{Metric tree $T_t$}\label{fig:MetricTree}
\end{center}
\end{figure}

The subset $\{p_i\}_{i=1}^{\infty} \subset T_t$ is an infinite ultrametric set. Indeed, if $j>i$ then $d_T(p_i,p_j)=t_i+t_j+(t_i-t_j)=2t_i$. Therefore, for any $k>j>i$, $d_T(p_i,p_j)=d_T(p_i,p_k)=2t_i$ and $d_T(p_j,p_k)=2t_j$. 

It follows that each such metric tree $(T_t,d_T)$ is $\SRA(0)$ full. 

Notice that the metric tree $(T_t,d_T)$ bi-Lipschitz embeds into $\ell^1(\N)$. If $\{e_i\}_{i\in\N}$ is the canonical basis of $\ell^1(\N)$, one can isometrically embed each $C_i$ into $\R e_i$.
See Proposition \ref{prop:sequence-embeddings} for a precise characterization of trees $(T_t,d_T)$ which bi-Lipschitz embed into a finite-dimensional Euclidean space.
\end{example}

\begin{example}[Laakso graphs]\label{example:laakso}
In \cite[Proposition 33]{LOZ}, the authors constructed infinite subsets of the Laakso graph satisfying the $\SRA(\tfrac35)$ condition. We improve upon this result by constructing infinite ultrametric subsets. The Laakso graph was constructed by Lang-Plaut \cite{LaPl} as a modification of Laakso spaces \cite{La1}; see also \cite{Laa} for a variant construction. This space has been intensively used as a motivating example in the theory of analysis in metric spaces, see e.g.\ \cite{GNRS}, \cite{NPSS}, \cite{LMN}, \cite{DaSc}, \cite{Ty} for a partial list of references.

\begin{figure}[h]
\begin{center}
\includegraphics[width=12cm]{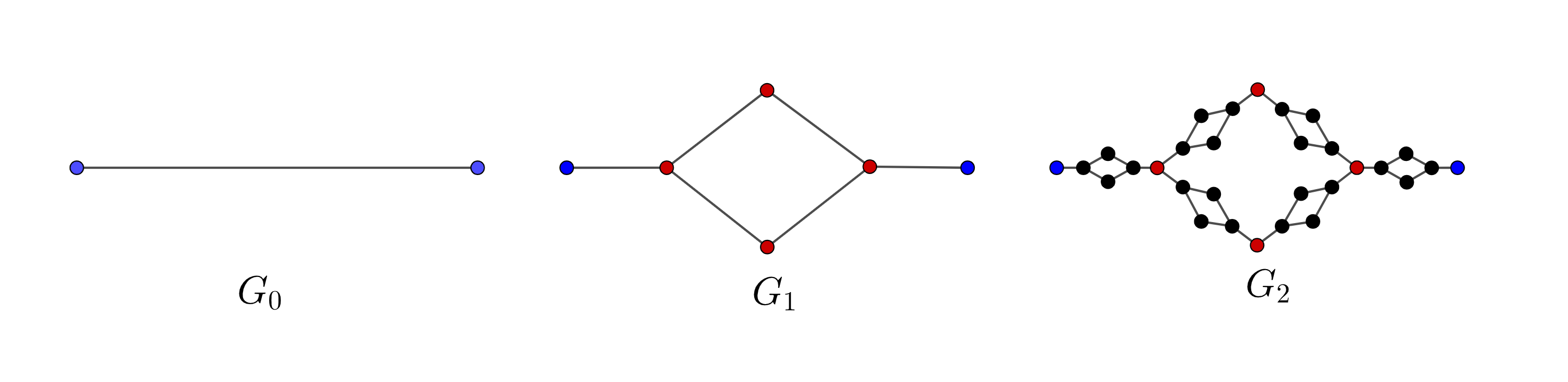}
\caption{Metric graphs $G_0$, $G_1$ and $G_2$}\label{fig:Laakso}
\end{center}
\end{figure}

We define the Laakso graph $G_\infty$ as the inverse limit of the system
\begin{equation}\label{eq:ILS}
G_0 \stackrel{\pi_0}{\longleftarrow} G_1 \stackrel{\pi_1}{\longleftarrow} G_2 \stackrel{\pi_2}{\longleftarrow} \cdots
\end{equation}
of metric graphs $G_m$ shown in Figure \ref{fig:Laakso}. For each $m$, $G_m$ is equipped with the path metric $d_m$ so that the diameter of $G_m$ is equal to one. In particular, $G_0 = [0,1]$, $G_1$ is the metric graph indicated in Figure \ref{fig:Laakso}, and $G_{m+1}$ is obtained from $G_m$ by replacing each edge $e$ in $G_m$ by a scaled copy $\varphi_e(G_1)$ of $G_1$ with scaling factor $4^{-m}$. The projection maps are defined as follows. First, $\pi_0:G_1 \to G_0$ is defined by the condition that each of the two geodesic paths in $G_1$ from $0$ to $1$ is mapped isometrically to the directed line segment $[0,1] = G_0$. Then $\pi_m:G_{m+1} \to G_m$ is defined, on each $\varphi_e(G_1) \subset G_{m+1}$, by conjugating the action of $\pi_0|_{G_1}$ by $\varphi_e$. The collection \eqref{eq:ILS} is an inverse limit system of metric graphs; see \cite[Section 2.2]{CK} for an introduction to this topic. The {\em Laakso graph $G_\infty$} is the inverse limit of the system \eqref{eq:ILS}, see \cite[Section 2.4]{CK}. In particular, $G_\infty$ is equipped with a well-defined metric $d_\infty$ so that
$$
d_\infty(q_0 \leftarrow q_1 \leftarrow q_2 \leftarrow \cdots,q_0' \leftarrow q_1' \leftarrow q_2' \leftarrow \cdots) = \lim_{m\to\infty} d_m(q_m,q_m')
$$
and there exist projection maps $\pi_m^\infty:G_\infty \to G_m$ so that $\pi_m^\infty(q_0 \leftarrow q_1 \leftarrow q_2 \leftarrow \cdots) = q_m$. All of the maps $\pi_m$ and $\pi_m^\infty$ are $1$-Lipschitz.

In what follows, we prove that $G_\infty$ is $\SRA(0)$ full. For each $m \in \N$, we will construct an ultrametric subset $F_m$ with cardinality $2^m$ in the $m$th approximating graph $G_m$. The sets $F_m$ will converge in the inverse limit topology to a limit subset $F_\infty \subset G_\infty$ of infinite cardinality. Since the ultrametric property is stable under such convergence, $F_\infty$ is again ultrametric and hence provides an example of an infinite subset of $G_\infty$ satisfying the $\SRA(0)$ condition.

Let $F_m \subset G_m$ be a subset of cardinality $2^m$ obtained as $F_m = \pi_m^{-1}(q)$ for some fixed abscissa $q \in G_0 = [0,1]$. For example, in the graph $G_1$ we may choose $q = \frac12$ and in the graph $G_2$ we may choose $q = \frac{6}{16}$; see Figure \ref{fig:Laakso}. The branching structure of the Laakso graph induces a canonical enumeration $F_m = \{\bp_w:w \in \{0,1\}^m \}$ by $m$-tuples of binary digits. For $v,w \in \{0,1\}^m$, set $|v\wedge w| = \min\{i:v_i \ne w_i\}$. Then
\begin{equation}\label{eq:d-on-f-m}
d_m(\bp_v,\bp_w) = \frac{c(k)}{4^m} \qquad \mbox{if $k = m-1-|v\wedge w| \in \{0,1,\ldots,m-1\}$,}
\end{equation}
where $c(0) = 2$ and
$$
c(k) = \frac43(4^k-1) \qquad k \ge 1.
$$
For instance, when $m=1$ we have $d_1(\bp_0,\bp_1) = \tfrac12$, when $m=2$ we have
\begin{itemize}
\item $d_2(\bp_{00},\bp_{01}) = d(\bp_{10},\bp_{11}) = \tfrac18$ and 
\item $d_2(\bp_{0j},\bp_{1k}) = \tfrac14$ for all $j,k$,
\end{itemize}
and when $m=3$ we have
\begin{itemize}
\item $d_3(\bp_{w_1w_20},\bp_{w_1w_21}) = \tfrac1{32}$ for all $w_1,w_2$, 
\item $d_3(\bp_{w_1jw_3},\bp_{w1kw_3'}) = \tfrac1{16}$ for all $j,k$ and all $w_1,w_3,w_3'$, and 
\item $d_3(\bp_{0jw_3},\bp_{1kw_3'}) = \tfrac5{16}$ for all $j,k$ and all $w_3,w_3'$.
\end{itemize}
We claim that $(F_m,d_m)$ is an ultrametric space. To see this, observe that \eqref{eq:d-on-f-m} implies that $d_m|_{F_m}$ aligns with the representation of $F_m$ as a rooted binary tree of level $m$. Vertices of this tree correspond to elements of $F_1 \cup \cdots \cup F_{m-1}$ (considered as initial segments of elements in $F_m$), and the distance between points $\bp_v$ and $\bp_w$, $v,w \in \{0,1\}^m$, is $\frac43(4^{m-1-|v\wedge w|}-1) 4^{-m}$. Thus $(F_m,d_m)$ is an ultrametric space.

The sequence $(G_m,d_m)$ Gromov-Hausdorff converges to $(G_\infty,d_\infty)$ \cite[Proposition 2.17]{CK}. Let $F_\infty$ denote a Gromov--Hausdorff limit of the sets $(F_m)$. Since the ultrametric condition passes to Gromov-Hausdorff limits, $(F_\infty,d_\infty)$ is ultrametric. Since $\#F_m \to \infty$, $F$ is an infinite set.
\end{example}

\begin{remark}\label{rem:Laakso-remark}
We can obtain a larger ultrametric subset of the Laakso graph as follows. For each $m\ge 2$, we may find two isometric copies of the set $F_m$ defined above, whose mutual distance is at least as large as the diameter of $F_m$. Thus we may choose $F_m' := \pi_m^{-1}(q) \cup \pi_m^{-1}(q')$. For instance, in the graph $G_2$ we may choose $q = \frac{6}{16}$ and $q' = \frac{10}{16}$; see Figure \ref{fig:Laakso}. Then $F_m'$ is again an ultrametric set. However, no ultrametric subset of $G_\infty$ can contain any three distinct points along any horizontal geodesic joining the endpoints of $G_0$.
\end{remark}

We also recall that all $\SRA(\alpha)$ free spaces are doubling, as was proved by Lebedeva, Ohta, and Zolotov \cite[Theorem 6]{LOZ}.

\begin{theorem}[Lebedeva--Ohta--Zolotov]\label{th:LOZ}
If there exists $\alpha\in(0,1)$ such that $(X,d)$ is $\SRA(\alpha)$ free then $(X,d)$ is doubling.
\end{theorem}

\begin{remark}
The converse of Theorem \ref{th:LOZ} is not true. In Example \ref{example:metrictree}, if one considers the sequence $t$ for which $t_i=\frac{1}{2^{i-1}}$, then one obtains a doubling metric tree $(T_t,d_T)$ containing an infinite ultrametric set. Thus $(T_t,d_T)$ is not $\SRA(\alpha)$ free for any $\alpha\in(0,1)$. Other examples of doubling spaces which are $\SRA(\alpha)$ full for some $\alpha\in(0,1)$ include the Heisenberg group and the Laakso graph.
\end{remark}

\begin{remark}\label{rem:sra-and-embeddings}
As a counterpoint to Proposition \ref{prop:SRAembeddability} we recall the following theorem of Zolotov, proved in \cite{Z2}.

\begin{theorem}\label{th:Z-embedding}
For each $0<\alpha<1$ and each $\SRA(\alpha)$-free metric space $(X,d)$, there exist constants $N = N(\alpha,X) \in \N$ and $L = L(\alpha,X) \ge 1$ and an $L$-bi-Lipschitz embedding of $(X,d)$ into $\R^N$. Moreover, $N$ and $L$ depend only on the maximal cardinality of an $\SRA(\alpha)$ subset of $X$.
\end{theorem}

Taken together, Proposition \ref{prop:SRAembeddability} and Theorem \ref{th:Z-embedding} imply that if $(X,d)$ is either $\SRA(\alpha)$-free or satisfies the $\SRA(\alpha)$ condition for some $\alpha<1$, then $X$ embeds bi-Lipschitzly into some finite-dimensional Euclidean space. While these two conditions stand at opposite extremes (one hypothesis asserts that all triplets in $X$ satisfy the $\SRA(\alpha)$ condition, while the other hypothesis states that there is a fixed upper bound on the cardinality of any subset of $X$ with the $\SRA(\alpha)$ condition), we do not see a path here to resolve the longstanding open problem of characterizing those metric spaces which admit finite-dimensional Euclidean bi-Lipschitz embeddings. For example, the metric trees considered in Example \ref{example:metrictree} are all $\SRA(0)$-full and do not fall into either of the previous two categories ($\SRA(\alpha)$-free or satisfying the $\SRA(\alpha)$ condition for some $\alpha<1$). However, Proposition \ref{prop:sequence-embeddings} shows that bi-Lipschitz embeddability and the doubling property depend on the rate of convergence of the defining sequence $(t_i)$.
\end{remark}

\subsection{Subsets of Euclidean snowflakes}

Lemma \ref{lem:bestSRA} asserts that if $(X,d^\alpha)$ is a snowflaked metric space, then $(X,d^\alpha)$ satisfies the $\SRA(2^\alpha-1)$ condition. In general, a space satisfying the $\SRA(\alpha)$ condition may not satisfy the $\SRA(\eps)$ condition for any choice of $\eps<\alpha$. The following theorem and its corollaries show that, in certain cases, we can ensure the existence of large (even infinite, or even of positive Hausdorff dimension) subsets of snowflaked metric spaces which verify the $\SRA(\eps)$ condition for smaller choices of $\eps$. 

In what follows, $d_E$ denotes the Euclidean metric.

\begin{theorem}\label{th:SRA-subsets-of-snowflake-spaces}
Fix $0<\eps<\alpha<1$. Then there exists a sequence $A := \{a_m:m \ge 1\}$ of real numbers so that $(A,d_E^\alpha)$ satisfies the $\SRA(\eps)$ condition.
\end{theorem}

In fact, $(a_m)$ is a geometric sequence: $a_m = (a_1)^m$ for all $m$. It follows that the set $A$ has Assouad dimension zero, and hence also has box-counting and Hausdorff dimension equal to zero.

\begin{proof}
Given $\eps<\alpha$ as in the statement of the theorem, set
\begin{equation}\label{eq:g}
g(x) := (1-x)^\alpha + \eps x^ \alpha - 1, \qquad 0 \le x \le 1.
\end{equation}
Then $g(0)=0$, $g(1)<0$, $g$ has a unique extremum in the interval $(0,1)$ at
$$
x_0: = \frac{\eps^{1/(1-\alpha)}}{1+\eps^{1/(1-\alpha)}},
$$
and 
$$
g''(x_0) = -\alpha(1-\alpha)\eps^{-1/(1-\alpha)} \biggl( 1 + \eps^{1/(1-\alpha)} \biggr)^{3-\alpha} < 0.
$$
Hence $g(x)$ has a unique maximum in the interval $(0,1)$ at $x=x_0$, and it follows that $g(x)$ has a unique root $a = a(\eps,\alpha)$ in the interval $(0,1)$. 

Set $a_1 := a$ and $a_m := a^m$ for all $m \ge 2$. It suffices for us to prove that
\begin{equation}\label{eq:SRA-to-check-0}
(a_i-a_m)^\alpha \le (a_i-a_j)^\alpha + \eps (a_j-a_m)^\alpha, \qquad \forall \, 0 < i < j < m.
\end{equation}
Since $(a_j)$ is a geometric sequence, the case $i>0$ is easily reduced to the case $i=0$ via scaling. Consequently, it is enough for us to verify that
\begin{equation}\label{eq:SRA-to-check}
(1-a_m)^\alpha \le (1-a_j)^\alpha + \eps (a_j-a_m)^\alpha, \qquad \forall \, 1\le j < m.
\end{equation}
For $j \ge 1$ define
$$
f_j(x) := (1-a_j)^\alpha + \eps (a_j-x)^\alpha - (1-x)^\alpha, \qquad 0 \le x \le a_j.
$$
Then $f_j(a_j) = 0$, $f_j(0) = g(a_j)  \ge 0$, and $f_j$ has a unique extremum at
$$
y_j := \frac{a_j - \eps^{1/(1-\alpha)}}{1-\eps^{1/(1-\alpha)}}.
$$
Clearly $y_j < a_j$ for all $j$, however, we have $y_j>0$ if and only if $a_j > \eps^{1/(1-\alpha)}$ i.e.
\begin{equation}\label{eq:j}
j < \frac{1}{1-\alpha} \frac{\log(1/\eps)}{\log(1/a)}.
\end{equation}
If \eqref{eq:j} holds then
$$
f''(y_j) = - \alpha(1-\alpha) \eps^{-1/(1-\alpha)} \frac{(1-\eps^{1/(1-\alpha)})^{3-\alpha}}{(1-a_j)^{2-\alpha}} < 0
$$
and so $f_j(x)$ has a unique maximum in the interval $(0,a_j)$ at $x=y_j$. If \eqref{eq:j} does not hold then $f_j$ is monotonic in $(0,a_j)$ with a maximum at $x=0$. Regardless of whether or not \eqref{eq:j} holds, we conclude that $f_j(x)>0$ for all $0<x<a_j$, in particular, $f_j(a_m)>0$ for all $m \ge j$. This establishes \eqref{eq:SRA-to-check}, and the other two statements in the definition of the $\SRA(\eps)$ condition are trivial. We therefore conclude that $A = \{a_m:m \ge 1\}$ is an $\SRA(\eps)$ set.
\end{proof}

Now, we present some immediate corollaries of Theorem \ref{th:SRA-subsets-of-snowflake-spaces}. Corollary \ref{cor:main-cor-2} restates Theorem \ref{thm:main-cor-2} from the introduction.

\begin{corollary}
Let $0<\alpha<1$. Then $(\R,d_E^\alpha)$ is $\SRA(\eps)$ full for each $0<\eps<\alpha$.
\end{corollary}

\begin{remark}
Observe that the conclusion fails dramatically when $\eps=0$: the largest ultrametric set in $(\R,d_E^\alpha)$ for $0<\alpha<1$ contains at most two points. To see this, let $x,y,z\in\R$ with $x<y<z$, and assume without loss of generality that $x=0$. If $\{x,y,z\}$ is ultrametric, then we must have
\[
|z|^{\alpha}\leq \max\{|y|^{\alpha},|z-y|^{\alpha}\},
\]
but this contradicts the fact that $f(x)=x^{\alpha}$ is an increasing function.
\end{remark}

\begin{corollary}\label{cor:main-cor-2}
Let $0<\alpha<1$ and let $(X,d)$ be a metric space containing a nontrivial geodesic. Then $(X,d^\alpha)$ is $\SRA(\eps)$ full for each $0<\eps<\alpha$.
\end{corollary}

We now improve Theorem \ref{th:SRA-subsets-of-snowflake-spaces}; instead of finding a sequence with the $\SRA(\eps)$ condition we find a subset of positive Hausdorff dimension with that property. For given $0<\eps<\alpha<1$, let $a = a(\eps,\alpha)$ be the unique root in $(0,1)$ of the function $g(x)$ in \eqref{eq:g}, as discussed in the proof of that theorem. If $a<\tfrac12$, let $C_a$ denote the self-similar Cantor set in $[0,1]$ obtained as the invariant set for the iterated function system $f_1(x) = ax$, $f_2(x) = 1-a+ax$. Note that the sequence $A$ defined in the proof of Theorem \ref{th:SRA-subsets-of-snowflake-spaces} is contained in $C_a$. Moreover, $C_a$ has Hausdorff dimension $\log(2)/\log(1/a)$.

\begin{proposition}\label{prop:cantor-set-subsets}
Let $0<\alpha<1$ and choose $\eps$ satisfying 
$$
0<\eps < 2^\alpha-1 < \alpha.
$$
Let $a=a(\eps,\alpha)<\tfrac12$ be defined as above, and let $C_a$ be the corresponding self-similar Cantor set. Then $(C_a,d_E^\alpha)$ satisfies the $\SRA(\eps')$ condition for
\begin{equation}\label{eq:eps-prime}
\eps' = \frac{(1-a)^\alpha - (1-2a)^\alpha}{a^\alpha} \, .
\end{equation}
Moreover, $\eps < \eps' < \alpha$ and $\eps'\to 0$ as $\eps\to 0$ for fixed $\alpha>0$.
\end{proposition}

Observe from \eqref{eq:g} that $\eps$ and $a$ are related by the equation
\begin{equation}\label{eq:eps}
\eps = \frac{1-(1-a)^\alpha}{a^\alpha}.
\end{equation}

\begin{proof}
It suffices to prove that for any $x,y,z \in C_a$ with $x<y<z$, the inequality
$$
|x-z|^\alpha \le \max \{ |x-y|^\alpha + \eps' |y-z|^\alpha , \eps' |x-y|^\alpha + |y-z|^\alpha \}
$$
holds true.

Denote by $C_a^L := C_a \cap [0,a]$ and $C_a^R := C_a \cap [1-a,1]$ the two self-similar pieces of $C_a$. By scaling invariance, we can assume that all three points $x,y,z$ are not contained entirely in $C_a^L$, nor are all three points contained in $C_a^R$. In view of the ordering of these points, this means that we must have $x \in C_a^L$ and $z \in C_a^R$. We will show the following two claims:
\begin{equation}\label{eq:claim1}
y \in C_a^L \qquad \Rightarrow \qquad |x-z|^\alpha \le \eps' |x-y|^\alpha + |y-z|^\alpha
\end{equation}
and
\begin{equation}\label{eq:claim2}
y \in C_a^R \qquad \Rightarrow \qquad |x-z|^\alpha \le |x-y|^\alpha + \eps' |y-z|^\alpha \, .
\end{equation}
In fact, \eqref{eq:claim2} follows by symmetry once we have established \eqref{eq:claim1}.

We now turn to the proof of \eqref{eq:claim1}. Our assumptions guarantee that
$$
0 \le x < y \le a \qquad \mbox{and} \qquad 1-a \le z \le 1
$$
and our goal is to prove that
\begin{equation}\label{eq:xyz}
\frac{(z-x)^\alpha - (z-y)^\alpha}{(y-x)^\alpha} \le \eps'.
\end{equation}
Fixing $x$ and $y$, we first observe that the function $z \mapsto ((z-x)^\alpha-(z-y)^\alpha)/(y-x)^\alpha$ is decreasing on $[1-a,1]$. It thus suffices for us to prove that
\begin{equation}\label{eq:xy}
G(x,y) := \frac{(1-a-x)^\alpha - (1-a-y)^\alpha}{(y-x)^\alpha} \le \eps' \qquad \forall \, 0 \le x < y \le a.
\end{equation}
An elementary computation gives
$$
\frac{\partial G}{\partial x} = \frac{\alpha(1-a-y)}{(y-x)^{1+\alpha}} \left( \frac1{(1-a-x)^{1-\alpha}} - \frac1{(1-a-y)^{1-\alpha}} \right)
$$
and
$$
\frac{\partial G}{\partial y} = \frac{\alpha(1-a-x)}{(y-x)^{1+\alpha}} \left( \frac1{(1-a-y)^{1-\alpha}} - \frac1{(1-a-x)^{1-\alpha}} \right) \, .
$$
Hence $\nabla G(x,y) \ne 0$ for any $(x,y) \in T := \{(x,y):0\le x < y \le a\}$ and $G$ is maximized on $\partial T$. Furthermore, $(\partial G/\partial x)(x,a)<0$ and $(\partial G/\partial y)(0,y)>0$ for all $0\le x<y\le a$ and $G(t-\delta,t+\delta)\to 0$ as $\delta\to 0$ for all $0<t<a$. It follows that
$$
\max_T G = \max_{\partial T} G = G(0,a) = \frac{(1-a)^\alpha - (1-2a)^\alpha}{a^\alpha} = \eps'.
$$
It follows from \eqref{eq:eps-prime}, the concavity of $t \mapsto t^\alpha$, and the bound $a \le \tfrac12$, that
$$
\eps' < \frac{\alpha a^{1-\alpha}}{(1-a)^{1-\alpha}} \le \alpha.
$$
Moreover, from \eqref{eq:eps-prime} and \eqref{eq:eps} we find that
$$
\lim_{a\to 0} \frac{\eps'(a)}{\eps(a)} = \lim_{a\to 0} \frac{(1-a)^\alpha-(1-2a)^\alpha}{1-(1-a)^\alpha} = 1
$$
and hence $\eps'\to 0$ as $\eps\to 0$.
\end{proof}

It is interesting to observe that all of the examples of subsets used to illustrate the $\SRA(\eps)$ fullness of $(\R,d_E^\alpha)$ as $\eps\to 0$ involve geometrically convergent constructions (e.g.\ geometric sequences or self-similar Cantor sets). In contrast, Example \ref{ex:arithmetic}, which highlights the sharpness of Lemma \ref{lem:bestSRA}, involves an arithmetic sequence. The distinction between arithmetic and geometric structure can also be witnessed in the following result.

\begin{proposition}\label{prop:positive-density}
Let $A \subset \R$ have positive Lebesgue density, and let $0<\alpha<1$. Then $(A,d_E^\alpha)$ does not satisfy the $\SRA(\eps)$ condition for any $0<\eps<2^\alpha-1$.
\end{proposition}

\begin{proof}
Let $a_0$ be a point of Lebesgue density for $A$, and choose a sequence $r_m \to 0$, $m \ge 2$, so that 
\begin{equation}\label{eq:A-upper-bound}
|A \cap (a_0-r_m,a_0+r_m)| > (1-\tfrac1m) (2r_m).
\end{equation}
Here we denote by $|E|$ the Lebesgue measure of a set $E \subset \R$. Set $\eps_m := r_m/m$. For each $m$, we claim that there exist points $b_m,c_m\in A$ with
\begin{equation}\label{eq:b}
b_m \in (a_0-\tfrac12 r_m - \eps_m,a_0 - \tfrac12 r_m + \eps_m)
\end{equation}
and
\begin{equation}\label{eq:c}
c_m \in (a_0 + \tfrac12 r_m - \eps_m,a_0 + \tfrac12 r_m + \eps_m).
\end{equation}
Indeed, if $A \cap (a_0-\tfrac12 r_m - \eps_m,a_0 - \tfrac12 r_m + \eps_m)$ is empty, then
\begin{equation}\label{eq:A-lower-bound}
|A \cap (a_0-r_m,a_0+r_m)| < 2r_m-2\eps_m
\end{equation}
which contradicts \eqref{eq:A-upper-bound}. Hence there exists $b_m$ as in \eqref{eq:b}, and a similar argument assures the existence of $c_m$ as in \eqref{eq:c}. For each $m$, the triple of points $b_m<a_m<c_m$ in $A$ satisfies
$$
|a_m-b_m| < \tfrac12 r_m + \tfrac1m r_m,
$$
$$
|a_m-c_m| < \tfrac12 r_m + \tfrac1m r_m,
$$
and
$$
|b_m-c_m| \ge (1-\tfrac2m) r_m.
$$
We see that the triple $(b_m,a_m,c_m)$ is {\em asymptotically arithmetic}, i.e., the arithmetic identity $a=(b+c)/2$ is approached in the limit as $m \to \infty$.

Suppose that $(A,d_E^\alpha)$ satisfies the $\SRA(\eps)$ condition. We will show that
\begin{equation}\label{eq:eps-ge-2-alpha-1}
\eps \ge 2^\alpha - 1.
\end{equation}
Applying the $\SRA(\eps)$ condition for the triple $b_m,a_m,c_m$ gives
$$
|b_m-c_m|^\alpha \le \max \{ |a_m-b_m|^\alpha + \eps |a_m-c_m|^\alpha , \eps |a_m-b_m|^\alpha + |a_m-c_m|^\alpha \}
$$
whence
$$
\left(1-\frac2m\right)^\alpha \le (1+\eps)\left(\frac12 + \frac1m\right)^\alpha.
$$
Letting $m \to \infty$ yields \eqref{eq:eps-ge-2-alpha-1} and completes the proof.
\end{proof}

To conclude this section we characterize when a snowflaked metric space isometrically embeds into an $\SRA(\alpha)$ free space. In the special case when the target space is a finite-dimensional normed vector space, the following result is \cite[Corollary 2.2]{LRW}.

\begin{theorem}
Let $(X,d)$ be a metric space. Then there exists $\alpha<1$ so that $(X,d^{\alpha})$ isometrically embeds into an $\SRA(\alpha)$ free space if and only if $X$ is finite.
\end{theorem}

\begin{proof}
Assume there is an $\alpha<1$  such that $(X,d^{\alpha})$ isometrically embeds into an $\SRA(\alpha)$ free.metric space $Z$. Then $(X,d^{\alpha})$ is also $\SRA(\alpha)$ free, so there exists $N\in\N$ such that for $F=(X,d^{\alpha})\subset Z$, $\#F\leq N$. Thus $X$ is finite. 

Conversely, assume that $X$ is finite. By \cite{DM}, there exists $\alpha<1$ and there exists $n\in\N$ such that $(X,d^{\alpha})$ embeds isometrically in $\R^n$. Since $\R^n$ is $\SRA(\alpha)$ free, the conclusion follows.
\end{proof}

\subsection{Bi-Lipschitz embeddability of countable ultrametric spaces with one limit point}

The following proposition characterizes the doubling and bi-Lipschitz embeddability properties of metric trees (as in Example \ref{example:metrictree}). Additionally, we characterize these properties for the compact countable ultrametric spaces with a single limit point from which these trees are defined and which also satisfy condition \eqref{eq:J-condition}. We begin with some notation. Let 
$$
(Y,d)=(\{p_{\infty},\ldots,p_2,p_1\},d)
$$
be a compact countable ultrametric space, where $p_j \to p_\infty$ and $p_{\infty}$ is the unique limit point of $\{p_j\}_{j=1}^{\infty}$. Reordering the sequence $\{p_j\}$ as necessary, we may assume that the function $j \mapsto d(p_j,p_\infty)$ is non-increasing. 

Let $\{t_k\}$ be the maximal strictly decreasing subsequence of $\{\tfrac12 d(p_j,p_\infty)\}_{j=1}^\infty$, i.e., $t_1>t_2>\cdots$ and $\{t_k:k=1,2,\ldots\} = \{ \tfrac12 d(p_j,p_\infty):j=1,2,\ldots \}$. Let $(T_t,d_T)$ be the metric tree associated to the sequence $\{t_k\}$ as in Example \ref{example:metrictree}. 

Finally, let $Y' = \{p_k':k=1,2,\ldots\} \subset Y$, where for each $k$ we have $\tfrac12 d(p_k',p_\infty) = t_k$. Observe that $(Y',d)$ isometrically embeds into $(T_t,d_T)$ as the set of endpoints of the vertical line segments, see the discussion in Example \ref{example:metrictree}. 

Observe that if $p_i,p_j \in Y$ are two points so that $d(p_i,p_\infty) = d(p_k',p_\infty)$ and $d(p_j,p_\infty) = d(p_\ell',p_\infty)$ with $k \ne \ell$, then
\begin{equation}\label{eq:d-p-i-p-j}
d(p_i,p_j) = \max\{2t_k,2t_\ell\}.
\end{equation}
Indeed, since $d(p_k',p_\infty) = 2t_k$ and $d(p_\ell',p_\infty) = 2t_\ell$, the ultrametric condition implies that $d(p_i,p_j) \le \max\{2t_k,2t_\ell\}$. On the other hand, if we assume without loss of generality that $k<\ell$, then the ultrametric condition also implies that 
\begin{equation}\label{eq:d2}
2t_k = d(p_k',p_\infty) \le \max\{d(p_i,p_j),d(p_\ell',p_\infty)\} = \max\{d(p_i,p_j),2t_\ell\}
\end{equation}
Since $t_k > t_\ell$ we conclude from \eqref{eq:d2} that $d(p_i,p_j) \ge 2t_k$ and so \eqref{eq:d-p-i-p-j} holds true.

\begin{proposition}\label{prop:sequence-embeddings}
Let $(Y,d)$ be a compact countable ultrametric space with a single limit point $p_\infty$, and let $(T_t,d_T)$ be the associated metric tree as described in the previous paragraph. Then the following are equivalent:
\begin{enumerate}
\item[{\em (1)}] $(T_t,d_T)$ is doubling.
\item[{\em (2)}] $(T_t,d_T)$ bi-Lipschitz embeds into $\R^{n}$ for some $n \in \N$.
\item[{\em (3)}] There exists $0<\delta<1$ and $m\in\N$ such that $t_{k+m}\leq (1-\delta) t_k$ for all $k\in\N$.
\item[{\em (4)}] $\limsup_{k\to\infty} \sup\{m:t_{k+m}>\frac{1}{2}t_k\}<\infty$.
\end{enumerate}
In case any of the above conditions are satisfied, then $(T_t,d_T)$ has finite total length.

Moreover, if
\begin{equation}\label{eq:J-condition}
\sup_k \# \{ j : \frac12 d(p_j,p_\infty) = t_k \} < \infty
\end{equation}
then the above four conditions are also equivalent to
\begin{enumerate}
\item[{\em (5)}] $(Y,d)$ is doubling.
\item[{\em (6)}] $(Y,d)$ bi-Lipschitz embeds into $\R^{n}$ for some $n\in\N$.
\end{enumerate}
In case these conditions on $Y$ are satisfied, then $\ell(Y)=\sum_{j=1}^{\infty} d(p_j,p_{j+1})$ is finite.
\end{proposition}

Before turning to the proof, we give an example which shows that the extra condition \eqref{eq:J-condition} is needed for the equivalence with conditions (5) and (6).

\begin{example}\label{example:simplices}
Choose an increasing sequence of integers $N_1<N_2<\cdots$ and define a compact countable ultrametric space $Y = \{p_k\}_{k \in \N} \cup \{p_\infty\}$ with metric $d$ by setting
$$
d(p_j,p_k) = 2^{-\min\{\ell:j \wedge k \le N_\ell\}},
$$
where $j\wedge k = \min\{j,k\}$ and setting $p_\infty = \lim_{k\to\infty} p_k$. That is, $d(p_k,p_\ell) = 1$ if either $k$ or $\ell$ (or both) lie in $\{1,\ldots,N_1\}$,
$$
d(p_k,p_\ell) = \frac12 \quad \mbox{if either $k$ or $\ell$ lie in $\{N_1+1,\ldots,N_2\}$ and neither $k$ nor $\ell$ lie in $\{1,\ldots,N_1\}$,}
$$
$$
d(p_k,p_\ell) = \frac14 \quad \mbox{if either $k$ or $\ell$ lie in $\{N_2+1,\ldots,N_3\}$ and neither $k$ nor $\ell$ lie in $\{1,\ldots,N_2\}$,}
$$
and so on. The space $(Y,d)$ contains equilateral subsets of arbitrarily large cardinality, and hence does not bi-Lipschitz embed into any finite-dimensional Euclidean space. In particular, $(Y,d)$ is not doubling. Moreover, condition \eqref{eq:J-condition} is not satisfied, since $N_\ell \to \infty$.
\end{example}

\begin{proof}[Proof of Proposition \ref{prop:sequence-embeddings}]
We establish this result by proving the implications in the following diagram:
\begin{equation}\label{eq:diag1}
\begin{matrix}
(1) & & \Longrightarrow & & (3) \\
\big\Uparrow & & & & \big\Updownarrow \\
(2) & & \Longleftarrow & & (4) \\
\end{matrix}
\end{equation}
and in case \eqref{eq:J-condition} is satisfied, the further implications
\begin{equation}\label{eq:diag2}
\begin{matrix}
(2) & & & & (3) \\
\big\Downarrow & & & & \big\Uparrow \\
(6) & & \Longleftrightarrow & & (5) \\ 
\end{matrix}
\end{equation}

Note that the implications $(2) \Rightarrow (1)$ and $(6) \Rightarrow (5)$ are trivial. Moreover, the implication $(5) \Rightarrow (6)$ follows from the fact that there exists a bi-Lipschitz embedding of an ultrametric space $(X,d)$ into $\R^{n}$ for some $n\in\N$ if and only if $X$ is doubling \cite[Proposition 3.3]{LM}.

We now show that $(1) \Rightarrow (3)$, which we prove by contradiction. Assume that for each $\delta>0$ and $m \in \N$ there exists $k \in \N$ so that $t_{k+m} > (1-\delta) t_k$. Choose $\delta = \tfrac12$. Recall that $d(p_k',p_\ell') = 2\max\{t_k,t_\ell\}$ for all $k,\ell$. It follows that for all $\ell\ge k$,
$$
p_\ell' \in B_{d}(p_k',2t_k).
$$
Moreover, for all $\ell_1,\ell_2$ satisfying $k \le \ell_1,\ell_2 \le k+m$,
$$
d(p_{\ell_1}',p_{\ell_2}') = 2 \max \{t_{\ell_1},t_{\ell_2}\} \ge 2t_{k+m} > t_k.
$$
Thus we found a set of $m$ points with pairwise distances at least $t_k$ within a ball of radius $2t_k$. Since $m$ was arbitrary, we conclude that $(Y',d)$ is not doubling, and hence $(T_t,d_T)$ is also not doubling.

Note that the above proof also shows that $(5) \Rightarrow (3)$, since $Y'$ is a subset of $Y$.

Next, we verify that $(3) \Leftrightarrow (4)$. Assume that (3) is satisfied for some $\delta>0$ and $m \in \N$, and choose $h \in \N$ so that $(1-\delta)^h \le \tfrac12 < (1-\delta)^{h-1}$. Then $t_{k+mh} \le (1-\delta)^h t_k$ for all $k$. Thus for each $k$, 
$\sup \{ m':t_{k+m'} > \tfrac12 t_k \} \le m h - 1$ and consequently,
$$
\limsup_{k\to\infty} \sup\{m':t_{k+m'}>\frac{1}{2}t_k\} \le m \left( \frac{\log 2}{\log(1/(1-\delta))} + 1 \right) - 1.
$$
Conversely, suppose that (4) is satisfied. Then $\sup_k \sup \{ m : t_{k+m} > \tfrac12 t_k \} < \infty$. Choose $M$ so that 
\begin{equation}\label{eq:M-condition}
t_{k+M} \le \tfrac12 t_k \qquad \forall \, k.
\end{equation}
It follows that (3) holds true with $\delta = \tfrac12$ and $m = M$.

To complete the verification of the implications in \eqref{eq:diag1}, we show that $(4) \Rightarrow (2)$. Assuming that condition (4) holds, choose $M$ as in \eqref{eq:M-condition}. We define an embedding of $(T_t,d_T)$ into $\R^{M+1}$ (equipped with the $\ell_1$ metric) as follows. Let $\vec{e}_0,\ldots,\vec{e}_M$ be an orthonormal basis for $\R^{M+1}$. For a point $(x,0) \in T_t$ with $0\le x \le t_1$, we set
$$
F(x,0) = x \vec{e}_0,
$$
while for a point $(t_k,y t_k) \in T_t$ with $0\le y \le 1$ and $k \in \N$, we set
$$
F(t_k,y t_k) = t_k \vec{e}_0 + y t_k \vec{e}_{m_0},
$$
where $m_0 \in \{1,2,\ldots,M\}$ is the residue of $k$ mod $M$, i.e., $k = \ell M + m_0$ for some integer $\ell$.

Let $q_1 = (t_{k_1},y_1 t_{k_1})$ and $q_2 = (t_{k_2},y_2 t_{k_2})$ be two points lying on distinct vertical line segments in $T_t$. Assume without loss of generality that $k_1<k_2$. Then
\begin{equation}\label{eq:d-1-2}
d_T(q_1,q_2) = (t_{k_1} - t_{k_2}) + y_1 t_{k_1} + y_2 t_{k_2}.
\end{equation}
If the residues of $k_1$ and $k_2$ mod $M$ are distinct, then
$$
|F(q_1) - F(q_2)| = (t_{k_1} - t_{k_2}) + y_1 t_{k_1} + y_2 t_{k_2} = d_T(q_1,q_2).
$$
This holds, in particular, if $k_2 < k_1 + M$.

Assume therefore that $k_2 \ge k_1 + M$ and $k_1$ and $k_2$ have the same residue mod $M$, i.e., $k_1 = \ell_1 M + m_0$ and $k_2 = \ell_2 M + m_0$. Then $k_2 - k_1 \ge M$ and hence
\begin{equation}\label{eq:assumption0}
t_{k_2} \le \tfrac12 t_{k_1}.
\end{equation}
In this case,
\begin{equation}\label{eq:d-E-1-2}
|F(q_1) - F(q_2)| = (t_{k_1} - t_{k_2}) + |y_1 t_{k_1} - y_2 t_{k_2}|.
\end{equation}
It is clear from \eqref{eq:d-1-2} and \eqref{eq:d-E-1-2} that $|F(q_1)-F(q_2)| \le d_T(q_1,q_2)$. For the converse inequality, it suffices to prove that
$$
y_1 t_{k_1} + y_2 t_{k_2} \le C\bigl( (t_{k_1} - t_{k_2}) + |y_1 t_{k_1} - y_2 t_{k_2} | \bigr)
$$
for some absolute constant $C \ge 1$. We will show that the conclusion holds with $C = 4$.

Suppose first that
\begin{equation}\label{eq:assumption1}
y_1 t_{k_1} \ge y_2 t_{k_2}.
\end{equation}
Then the desired inequality is
$$
(C + C y_2 + y_2) t_{k_2} \le (C + C y_1 - y_1) t_{k_1}.
$$
In view of \eqref{eq:assumption1} it suffices to have
$$
C t_{k_2} \le (C-2y_1) t_{k_1}
$$
and in view of \eqref{eq:assumption0} and the fact that $0 \le y_1 \le 1$ it suffices to have
\begin{equation}\label{eq:CC}
\frac12 C \le C - 2
\end{equation}
and so the choice $C = 4$ suffices.

Next, suppose that 
\begin{equation}\label{eq:assumption2}
y_1 t_{k_1} \le y_2 t_{k_2}.
\end{equation}
Following a similar line of reasoning, we again arrive to \eqref{eq:CC} and conclude that $C = 4$ suffices.

We conclude that if condition (4) holds, then $F$ defines a $4$-bi-Lipschitz embedding of $(T_t,d_T)$ into $\R^{M+1}$ with the $\ell_1$ metric. Hence $(T_t,d_T)$ embeds $L$-bi-Lipschitzly into $\R^{M+1}$ for some $L = L(M)$.

Assuming now that one of the conditions (1) -- (4) is satisfied, we compute the length of the metric tree $T_t$ as $\ell(T_t) = t_1 + \sum_{k=1}^\infty t_k$. Since 
$$
\sum_k t_k = \sum_{\ell=0}^\infty \sum_{k_0=1}^m t_{\ell m + k_0} \le \sum_{\ell=0}^\infty (1-\delta)^\ell \sum_{k_0=1}^m t_{k_0} < \infty,
$$
we conclude that $T_t$ is rectifiable.

To finish the proof, we show that the remaining implication $(2) \Rightarrow (6)$ in \eqref{eq:diag2} holds under the extra assumption \eqref{eq:J-condition}. Thus assume that $F:(T_t,d_T) \to \R^N$ is an $L$-bi-Lipschitz embedding, and that $\#\{j : \tfrac12 d(p_j,p_\infty) = t_k \} \le J < \infty$ for all $k$. We will construct a bi-Lipschitz embedding $G:(Y,d) \to \R^{N+J-1}$. Considering only the subset $(Y',d)$ in $(T_t,d_T)$ we have that
$$
\frac1L d(p_k',p_\ell') \le |F(p_k') - F(p_\ell')| \le L d(p_k',p_\ell') \qquad \forall \, k,\ell.
$$
In particular, if $k<\ell$ then we have
$$
\frac2L t_k \le |F(p_k') - F(p_\ell')| \le 2L t_k \qquad \forall \, k,\ell.
$$
For each $k$, the sequence $Y$ contains an ultrametric subset $Z_k$ consisting of at most $J$ elements $p_j$ satisfying the condition 
\begin{equation}\label{eq:k-from-j}
d(p_j,p_\infty) = 2t_k.
\end{equation}
Let $\iota_k:Z_k \to \R^{J-1}$ be an isometric embedding; see Theorem \ref{th:finite-ultrametric-embeddability} for the existence of such a map. Applying a translation in $\R^{J-1}$ if necessary, we also assume that the point $p_k'$ (which occurs as an element in $Z_k$) is mapped by $\iota_k$ to the zero element of $\R^{J-1}$. 

We now define 
$$
G(p_j) = F(p_k') \oplus \tfrac1{8L} \iota_k(p_j)
$$
for $j=1,2,\ldots$, where $k = k(j)$ is defined as in \eqref{eq:k-from-j}, and we show that this map is the desired bi-Lipschitz embedding.

Now let $i$ and $j$ be any two indexes satisfying $d(p_i,p_\infty) = 2t_k$ and $d(p_j,p_\infty) = 2t_\ell$. If $k = \ell$ then
$$
|G(p_i) - G(p_j)| = |w_i-w_j| = d(p_i,p_j).
$$
We thus consider the case $k \ne \ell$, and we assume without loss of generality that $k<\ell$. Then
\begin{equation*}\begin{split}
|G(p_i) - G(p_j)| &\le |G(p_k') - G(p_\ell')| + |G(p_i) - G(p_k')| + |G(p_j) - G(p_\ell')| \\
&= |F(p_k') - F(p_\ell')| + \frac1{8L}d(p_i,p_k') + \frac1{8L}d(p_j,p_\ell') \\
&\le L d(p_k',p_\ell') + \frac1{8L} (2 t_k  +2 t_\ell) \\
&\le ( 2L + \frac1{2L} ) t_k = ( L + \frac1{4L} ) d(p_i,p_j);
\end{split}\end{equation*}
see \eqref{eq:d-p-i-p-j} for the final equality. On the other hand,
\begin{equation*}\begin{split}
|G(p_i) - G(p_j)| &\ge |F(p_k') - F(p_\ell')| - \frac1{8L} d(p_i,p_k') - \frac1{8L} d(p_j,p_\ell') \\
&\ge \frac1L d(p_k',p_\ell') - \frac1{2L} t_k  \\
&= \frac1{2L} t_k = \frac1{4L} d(p_i,p_j).
\end{split}\end{equation*}
We conclude that $G$ is a bi-Lipschitz embedding of $(Y,d)$ into $\R^{N+J-1}$. We leave to the reader the verification that $\ell(Y)$ is finite in this situation.

This completes the proof of Proposition \ref{prop:sequence-embeddings}.
\end{proof}

\section{Ordered sets and rough self-contracting curves}\label{sec:self-contracting}

\subsection{Definitions, background and motivation}

The following class of curves was studied in \cite{DaDuDe} (where the same curves were introduced but with reverse parametrization). We give a unified definition which covers both the case of curves and finite sets.

\begin{definition}[Rough $\lambda$-self-contracting curve]
Fix $-1 \le \lambda \le 1$ and let $I \subset \R$ be a compact subset. A map $\gamma$ from $I$ into a metric space $(X,d)$ is said to be a {\em rough $\lambda$-self-contracting curve fragment} if for every $t_{1}\leq t_{2}\leq t_{3}$ in $I$ we have
\begin{equation}\label{l-curve}
d(\gamma(t_{2}),\gamma(t_{3}))\leq d(\gamma(t_{1}),\gamma(t_{3}))+\lambda
d(\gamma(t_{1}),\gamma(t_{2})).
\end{equation}
\end{definition}

We typically consider the case when the parameterizing set $I$ is either an interval $[0,T]$ (in which case we call $\gamma$ a {\em rough $\lambda$-self-contracting curve}) or a finite ordered set $\{t_1<t_2<\cdots<t_N\}$ (in which case we call $\gamma$ a {\em rough $\lambda$-self-contracting finite ordered set}). We stress that $\gamma$ is in general {\bf not} assumed to be continuous. 

In what follows, we shall assume that all curves are injective mappings, as a non-injective rough $\lambda$-self-contracting curve would necessarily be locally constant. The set $\gamma(I)$ inherits a total order from 
$I$, defined as follows: if there exist $t,s\in I$ with $t< s$ such that $x=\gamma(t)$ and $y=\gamma(s)$ then we write $x<y$.

\smallskip

The case $\lambda=0$ corresponds to the original notion of {\em self-contracted} curves introduced in \cite{DLS}, while $\lambda=-1$ corresponds to the class of {\em geodesic sets}, i.e., subsets of $\R$ with the Euclidean metric. Inverting the orientation gives the analogous notion of {\em rough $\lambda$-self-expansion}, introduced in \cite{DaDuDe} under the name {\em $\lambda$-curves}, the case $\lambda=0$ corresponding to the class of {\em self-expanded} curves. A map $\gamma:I \to (X,d)$ is rough $\lambda$-self-contracting if and only if $\tilde\gamma:I \to (X,d)$ given by $\tilde\gamma(t) = \gamma(-t)$ is rough $\lambda$-self-expanding.

\smallskip 

Finally, we say that $\gamma$ is a {\em rough $\lambda$-self-monotone curve} if it is either rough $\lambda$-self-contracting or rough $\lambda$-self-expanding. One can also analogously define {\em rough $\lambda$-self-monotone sets}. Notice that the class of rough $\lambda$-self-monotone curve is invariant under time reversal.

\begin{example}
Let $x_1$ and $x_3$ be two points in $\R^2$ and let us determine where a third point, $x_2$, could be located in order that \eqref{l-curve} be satisfied. For simplicity, assume $x_1=(1,0)$, $x_3=(0,0)$ and $x_2=re^{i\theta}$. 
Then \eqref{l-curve} reads
\begin{equation}\label{eggeq}
r\leq 1+\lambda \sqrt{(r\cos\theta-1)^2+r^2\sin^2\theta}.
\end{equation}
In other words, the image of $[t_1,t_3]$ under a rough $\lambda$-self-expanding map $\gamma$ with $\gamma(t_1) = (0,0)$ and $\gamma(t_3) = (1,0)$ must lie in the region $A=\{(x,y) : \sqrt{x^2+y^2}\leq 1+\lambda \sqrt{(x-1)^2+y^2}\}$. 
Now let 
\[
X=\{(x,y)\in\R^2:\sqrt{x^2+y^2}= 1+\lambda \sqrt{(x-1)^2+y^2};y\geq 0\}\cup (a,0]\times\{0\}, \quad a = -\frac{1+\lambda}{1-\lambda};
\] 
see Figure \ref{fig:lambda-curve-figure} for the case $\lambda=\tfrac12$. 

\begin{figure}[h]
\includegraphics[width=5.5cm]{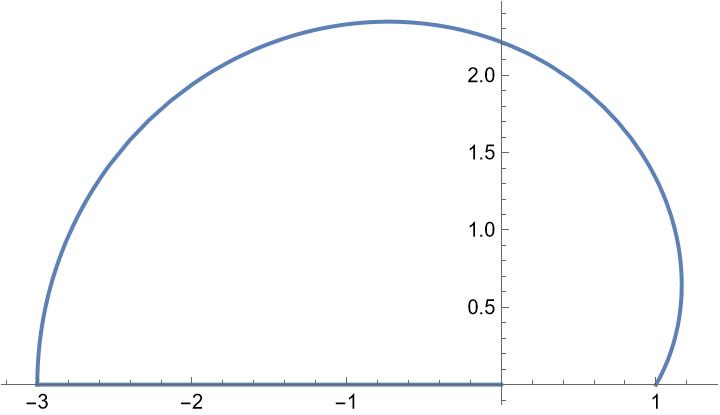}
\caption{The curve $\{(x,y):\sqrt{x^2+y^2} = 1 + \tfrac12 \sqrt{(x-1)^2+y^2}\} \cup [-3,0]\times\{0\}$}\label{fig:lambda-curve-figure}
\end{figure}
\end{example}

The class of rough $\lambda$-self-monotone curves can be related to other classes of curves in the literature. For instance, let us discuss the connection to the well-known class of bounded turning curves.

\begin{definition}[Bounded turning curves]
Let $\gamma:I\to X$ be a injective curve. We say that  $\gamma$ is {\em $M$-bounded turning} if, for any pair of points $x,y$ in $\gamma(I)$,
\[
\mbox{diam} \,\Gamma_{xy}\leq M d(x,y),
\]
where $\Gamma_{xy}$ denotes the arc connecting $x$ to $y$.
\end{definition}

\begin{lemma}\label{Lemma:expcurveBT}
Every rough $\lambda$-self-monotone curve is $2$-bounded turning for $-1\leq \lambda<0$ and $2\left(\tfrac{1+\lambda}{1-\lambda}\right)$-bounded turning for $0\leq\lambda<1$.
\end{lemma}

\begin{proof}
Without loss of generality, we can assume that the curve $\gamma$ is a $\lambda$-self-expanding curve.
Let $0\leq\lambda<1$. For $t_1\leq t_2\leq t_3$, the triangle inequality gives
\begin{equation*}
\lambda d(\gamma(t_{2}),\gamma(t_{3}))\leq \lambda d(\gamma(t_{1}),\gamma(t_{2}))+ \lambda d(\gamma(t_{1}),\gamma(t_{3})).
\end{equation*}
Summing this inequality with \eqref{l-curve} yields
\begin{equation}\label{eq:c-condition}
d(\gamma(t_{1}),\gamma(t_{2}))\leq \left(\dfrac{1+\lambda}{1-\lambda}\right)d(\gamma(t_{1}),\gamma(t_{3})).
\end{equation}
Let $c:=\dfrac{1+\lambda}{1-\lambda}$. Notice that
\[
d(s,t)\leq d(s,x)+d(x,t)\leq 2c d(x,z) \qquad \mbox{if $x<s<t<z$}
\]
by two applications of \eqref{eq:c-condition}. To finish the proof, observe that every rough $\lambda$-self-expanding curve for $-1\leq \lambda<0$ is self-contracting so it is $2$-bounded turning.
\end{proof}

For later purposes, we define 
$$
M_\lambda := \begin{cases} 2 \tfrac{1+\lambda}{1-\lambda}, & 0\le\lambda<1, \\ 2, & -1 \le \lambda < 0. \end{cases}
$$

Observe that the converse of Lemma \ref{Lemma:expcurveBT} does not hold. For an example with $\lambda=0$, let $\delta>0$ and let $X_{\delta}=\{e^{i\theta}\in\mathbb{C}: 0\leq\theta<\pi+\delta\}$. The curve $X_{\delta}$ is $2$-bounded turning but it is not self-expanding.

On the other hand, every $1$-bounded turning curve is both self-expanding and self-contracting.

\begin{remark}\label{rem:ultraorder}
Any compact metric space satisfying the $\SRA(0)$ condition (i.e., any compact ultrametric space) can be equipped with a linear order with respect to which it satisfies the $1$-bounded turning condition, and hence is both self-expanding and self-contracting. See, e.g., \cite[Lemma 2.3]{KMZ}. 
\end{remark}

\subsection{Main theorem}

In what follows, we prove Theorem \ref{thm:rectifiability}, which asserts the rectifiability of bounded, roughly $\lambda$-self-monotone curves in $\SRA(\alpha)$ free metric spaces for suitable choices of $\lambda$ and $\alpha$. For the reader's convenience we restate the theorem here, along with several remarks.

\begin{theorem}\label{thm:rectifiability2}
Let $\tfrac12<\alpha<1$ and let $(X,d)$ be an $\SRA(\alpha)$ free metric space. Then there exists $\lambda_0 = \lambda_0(\alpha,X)>0$ so that any bounded, rough $\lambda$-self-monotone curve in $(X,d)$, with $\lambda \in [-1,\lambda_0)$ is rectifiable. Furthermore, the length of $\gamma$ is at most $C' \diam(\gamma)$ for some constant $C' = C'(\alpha,X)$.
\end{theorem}

\begin{remark}\label{rem:rectifiability-remark}
We recall that Zolotov \cite{Z} proves that self-contracting curves ($\lambda=0$) in $\SRA(\alpha)$ free metric spaces are rectifiable whenever $\alpha>\tfrac12$. It remains unclear whether this result is sharp; specifically, whether the existence of an unrectifiable self-contracting curve in $(X,d)$ necessarily implies that $(X,d)$ is not $\SRA(\alpha)$ free for any $0\leq \alpha<1$. The proof which we give for Theorem \ref{thm:rectifiability2} closely follows Zolotov's argument.
\end{remark}

\begin{remark}\label{rem:dependence-of-constants-remark}
The dependence of $\lambda_0$ and $C$ on the metric space $X$ is only in terms of the maximal cardinality $K$ of any subset $S\subset X$ satisfying the $\SRA(\alpha)$ condition. See e.g.\ Remark \ref{EFremark} for estimates on the size of $K$ in the case when $X = \R^n$. When $X = \R^n$, the value of $\lambda_0$ which we obtain through our proof is substantially smaller than the best currently known value (e.g.\ $\lambda = \tfrac1n$) from \cite{DaDuDe}. However, the argument which we provide here uses only general metric properties of the underlying space, and does not rely on specifically Euclidean convex geometric results such as the Carath\'eodory theorem. As such it applies to a substantially broader array of ambient spaces.
\end{remark}

\begin{remark}\label{rem:rectifiability-remark2}
Observe that bounded unrectifiable self-monotone curves can be constructed in various spaces, including the Heisenberg group (Example \ref{example:heis}), Hilbert spaces (Example \ref{example:hilbert}), and some metric trees (Example \ref{example:metrictree}). In Hilbert spaces, such a curve can be constructed as shown in \cite[Example 3.5]{ST} or \cite[Example 2.2]{DDDR}. Additionally, the set $X = \{ (0,t) : 0 \le t \le T\} \subset \Heis^1$ is a self-contracting curve, as demonstrated in \cite[8.1]{LOZ}. By Remark \ref{rem:ultraorder}, self-monotone curves can also be constructed in metric trees or Laakso graphs, as they are $\SRA(0)$ full. 

One can easily check that the sequence of points $\{p_i\}_{i=1}^\infty$ in Example \ref{example:metrictree} forms a self-contracting curve.\footnote{Here $\{p_i\}_{i=1}^\infty$ is a non-continuous self-contracting curve and its length is defined as the length of the polygonal curve connecting the points.} However, the finiteness of the length of the curve depends on the defining sequence $t=\{t_i\}$. For instance, the sequence $t_i=\frac{1}{i}$ yields a curve of infinite length, while the sequence $t_i=\frac{1}{2^{i-1}}$ results in a curve of finite length. See, for example \cite[Example 31 and 32]{LOZ}. Furthermore, one can check that in the first case, $T_t$ does not admit a bi-Lipschitz embedding into Euclidean space, while in the second case it does. See also Proposition \ref{prop:sequence-embeddings}.

Whether or not there exists a bounded unrectifiable self-monotone curve in the Laakso graph remains an interesting open problem, see the discussion surrounding Question \ref{Q:Laakso-existence} for more details.
\end{remark}

\begin{remark}
Theorem \ref{thm:rectifiability2} also applies to semi-globally $\SRA(\alpha)$ free spaces. A metric space $(X,d)$ is said to be {\em semi-globally $\SRA(\alpha)$ free}, for $0\le \alpha < 1$, if for each $x\in X$ and $r>0$ there exists $N_{r,x}\in\N$ such that for each $F\subset B(x,r_x)$, if $F$ satisfies the $\SRA(\alpha)$ condition then $\#F\leq N_{r,x}$. The class of semi-globally $\SRA(\alpha)$ free spaces (see \cite[Theorem 2]{LOZ}) includes the following:
\begin{itemize}
\item Finite-dimensional Alexandrov spaces of curvature $\geq k$ with $k\in\R$.
\item Complete, locally compact Busemann NPC spaces (e.g., CAT(0)-spaces) with
locally extendable geodesics.
\end{itemize}
\end{remark}

\subsection{Proof of Theorem \ref{thm:rectifiability}}

The first step in the proof is a reduction from curves to finite ordered sets. To this end, we introduce discrete analogs for the diameter and length of a curve $\gamma$.

\begin{definition}[Discrete diameter and discrete length]\label{def:discrete-diameter-length}
Let $t_1<\cdots<t_m$ be a finite set of real numbers, and let $S = \gamma(\{t_1,\ldots,t_m\}) \subset (X,d)$ be a discrete ordered set, with $x_i = \gamma(t_i)$. Define the {\em discrete length} of $S$ to be
$$
L(S) := \sum_{i=1}^{m-1} d(x_i,x_{i+1})
$$
and the {\em discrete diameter} of $S$ to be
$$
D(S) := d(x_1,x_m).
$$
\end{definition}

The terminology `discrete diameter' is motivated by the case of rough $\lambda$-self-monotone sets. Note that if $S$ is a rough $\lambda$-self-monotone set, then $\diam(S) \le M_\lambda D(S) \le M_\lambda \diam(S)$, where $M_\lambda$ is defined just after the proof of 
Lemma \ref{Lemma:expcurveBT}.

\begin{remark}\label{rem:continuous-to-discrete}
Note that if $\gamma:I \to X$ is a bounded, unrectifiable curve and $C>0$ is any positive constant, then there exists a discrete ordered subset $S \subset \gamma$ with $L(S) > C D(S)$. Stated another way, if $\gamma$ is bounded and $L(S) \le C D(S)$ for all discrete ordered subsets $S$ of $\gamma$, then $\gamma$ is rectifiable and has length at most $C\diam(\gamma)$.
\end{remark}

We now state a discrete analog of Theorem \ref{thm:rectifiability}.

\begin{theorem}\label{thm:discrete-rectifiability}
For each $\tfrac12<\alpha<1$ and $K \in \N$, there exists $\lambda_0=\lambda_0(\alpha,K)>0$ and $C = C(\alpha,K)>0$ so that if $(\gamma,d)$ is a rough $\lambda$-self-monotone finite ordered set, $-1\le\lambda<\lambda_0$, satisfying $L(\gamma) > C D(\gamma)$, then $(\gamma,d)$ contains a $K$ element subset $(A,d)$ satisfying the $\SRA(\alpha)$ condition.
\end{theorem}

We now show how to reduce Theorem \ref{thm:rectifiability} to Theorem \ref{thm:discrete-rectifiability}.

\begin{proof}[Proof of Theorem \ref{thm:rectifiability}]
Assume that $(X,d)$ is an $\SRA(\alpha)$ free metric space, and fix an integer $K$ so that $X$ contains no $\SRA(\alpha)$ subset of cardinality $K$. Choose $\lambda_0$ as in the statement of Theorem \ref{thm:discrete-rectifiability} and assume that $\gamma$ is a bounded and unrectifiable rough $\lambda$-self-monotone curve $\gamma$ for some $-1\le\lambda<\lambda_0$. Observe that if a rough $\lambda$-self-monotone curve consists of a finite set of points, it is trivially rectifiable. Thus, without loss of generality, we may assume that the curve contains an infinite set of points. Let $C$ be the constant guaranteed by Theorem \ref{thm:discrete-rectifiability}. Since $\gamma$ is bounded and unrectifiable, we can choose a finite ordered set $S \subset \gamma$ so that $L(S) > C D(S)$. Then $S$ contains an $\SRA(\alpha)$ subset of cardinality $K$, but this contradicts the choice of $K$. 

We now turn to the proof of the final claim of the theorem. Assume that $\gamma \subset X$ is a bounded and rectifiable rough $\lambda$-self-monotone curve with $\lambda < \lambda_0(\alpha,K)$ where $K$ is as above. Then $\gamma$ contains no subset satisfying the $\SRA(\alpha)$ condition of cardinality $K+1$. It follows that
$$
L(S) \le C(\alpha,K+1) D(S)
$$
for every rough $\lambda$-self-expanding finite ordered subset $S \subset \gamma$. By Remark \ref{rem:continuous-to-discrete}, the length of $\gamma$ is at most $C(\alpha,K+1) \diam(\gamma)$.
\end{proof}

\begin{remark}
Note that the parameter $K$ depends on the metric space $(X,d)$. In \eqref{EFquantitative}, a bound for $K$ in the Euclidean case is given.
\end{remark}

We now discuss the relationship between the rough self-expanding and contracting conditions and the $\SRA$ condition in the setting of a finite ordered set $S$. Recall that the former conditions impose two restrictions on the mutual distances between three points, stated relative to the given ordering on $S$. In contrast, the $\SRA$ condition imposes restrictions on the mutual distances between three points in $S$, which must hold with respect to any permutation of the points. More precisely, if $x_1 < x_2 < x_3$ are points in $S$, then the rough $\lambda$-self-expanding condition asserts that
$$
d(x_1,x_2) \le d(x_1,x_3) + \lambda d(x_2,x_3)
$$
while the rough $\lambda$-self-contracting condition asserts that
$$
d(x_2,x_3) \le d(x_1,x_3) + \lambda d(x_1,x_3).
$$
On the other hand, in order to check the $\SRA(\alpha)$ condition for this particular triple of points, we must verify all of the following inequalities:
\begin{eqnarray*}
&d(x_1,x_2) \le \max\{ d(x_1,x_3) + \alpha d(x_2,x_3), \alpha d(x_1,x_3) + d(x_2,x_3) \} \\
&d(x_1,x_3) \le \max\{ d(x_1,x_2) + \alpha d(x_2,x_3), \alpha d(x_1,x_2) + d(x_2,x_3) \} \\
&d(x_2,x_3) \le \max\{ d(x_1,x_2) + \alpha d(x_1,x_3), \alpha d(x_1,x_2) + d(x_1,x_3) \}.
\end{eqnarray*}

We assume that $\lambda \le \alpha$. In order to derive the $\SRA(\alpha)$ conclusion for such a triple, selected from a set which is both rough $\lambda$-self-expanding and $\lambda$-self-contracting, it suffices to verify that either
\begin{equation}\label{eq:fosra}
d(x_1,x_3) \le d(x_1,x_2) + \alpha d(x_2,x_3)
\end{equation}
or
\begin{equation}\label{eq:bosra}
d(x_1,x_3) \le \alpha d(x_1,x_2) + d(x_2,x_3)
\end{equation}
holds true. To simplify matters we work consistently with \eqref{eq:fosra}.

\begin{definition}
A finite ordered set $S$ in a metric space $(X,d)$ is said to satisfy the {\em medial ordered $\SRA(\alpha)$ condition} if the inequality \eqref{eq:fosra} holds true for any choice of $x_1<x_2<x_3$ in $S$.
\end{definition}

We record an elementary lemma which formalizes the preceding discussion.

\begin{lemma}\label{lem:elementary}
If $S \subset (X,d)$ is a finite ordered set which is both rough $\lambda$-self-expanding and $\lambda$-self-contracting and which satisfies the medial ordered $\SRA(\lambda)$ condition, then $(F,d)$ satisfies the $\SRA(\lambda)$ condition.
\end{lemma}

We will reduce the proof of Theorem \ref{thm:discrete-rectifiability} to the following two technical propositions. 
\begin{proposition}\label{prop:1}
Let $0<\theta<1$ and $m \in \N$. Then there exist $\lambda_1 = \lambda_1(\theta,m)>0$ and $C = C(\theta,m)>0$ so that if $S=\{x_1<x_2<\cdots<x_n\}$ is a finite, ordered, rough $\lambda$-self-monotone set, $-1\le \lambda < \lambda_1$,
with $L(S) > C D(S)$, then $n\ge m$ and there exists $F = \{x_{i_1}<\cdots<x_{i_m}\} \subset S$ so that $(F,d)$ satisfies the medial ordered $\SRA(\theta)$ condition.
\end{proposition}

\begin{proposition}\label{prop:2}
Let $0<\theta<1$ and let $M>1$. If
\begin{equation}\label{eq:alpha-lower-bound}
\alpha > \left( 1 - \frac1M \right) \frac{1+\theta}{1-\theta},
\end{equation}
then there exists an integer $p = p(\theta,M,\alpha)$ so that if $(Z,d)$, $Z = \{ z_1 < \cdots < z_p \}$, is a finite, ordered, $M$-bounded turning set satisfying the medial ordered $\SRA(\theta)$ condition, then 
\begin{equation}\label{eq:zzz}
d(z_{i+1},z_p) < d(z_i,z_p) + \alpha d(z_i,z_{i+1})
\end{equation}
for some $i$, $1\le i < p-1$.
\end{proposition}

We now show how to derive Theorem \ref{thm:discrete-rectifiability} as a consequence of Propositions \ref{prop:1} and \ref{prop:2}. From now on, without loss of generality, we assume that our curve or set is roughly $\lambda$-self-expanding. The proof makes use of Ramsey's theorem for $3$-uniform hypergraphs. The relevant Ramsey number $R_3(p,K)$ is the least integer $m$ such that for any $2$-coloring of the set of all unordered triples of elements in a set $S$ of $m$ elements, there exists either a subset $T_1$ of $S$ of size $p$ such that all unordered triples formed by elements of $T_1$ are monochromatic in the first color, or a subset $T_2$ of $S$ of size $K$ such that all unordered triples formed by elements of  $T_2$ are monochromatic in the second color. 

The upper bound
\begin{equation}\label{eq:CFS}
R_3(p,K) \le \exp(c K^{p-2} \log K),
\end{equation}
where $c$ is an absolute constant, is due to Conlon--Fox--Sudakov \cite{CFS}, see also Sudakov's 2010 ICM proceedings article \cite{S}. For our purposes, it is enough to know that $R_3(p,K)$ is finite, although we do make use of \eqref{eq:CFS} to give an explicit upper bound for the parameter $\lambda_0$ in Theorem \ref{thm:rectifiability}.

\begin{proof}[Proof of Theorem \ref{thm:discrete-rectifiability}]
Let $\tfrac12 < \alpha < 1$ and $K \in \N$ be given. The value $\lambda_0$ will satisfy several constraints, the first of which is
\begin{equation}\label{eq:lambda-first}
\lambda_0 < \frac13(2\alpha-1).
\end{equation}
See also \eqref{eq:lambda-second}. We note that \eqref{eq:lambda-first} is equivalent to
$$
2 \frac{1+\lambda_0}{1-\lambda_0} \le \frac{1+\alpha}{1-\alpha/2} =: M_*.
$$
If $\gamma$ is a rough $\lambda$-self-expanding set with $\lambda \le \lambda_0$, then $\gamma$ is $M_\lambda$-bounded turning with $M_\lambda$ as in Lemma \ref{Lemma:expcurveBT}. Moreover,
$$
M_\lambda \le 2 \frac{1+\lambda_0}{1-\lambda_0} \le M_*.
$$
Furthermore,
$$
1 - \frac1{M_*} = \frac{\tfrac32\alpha}{1+\alpha}
$$
and
$$
\alpha > 1 - \frac1{M_*}
$$
since $\alpha>\tfrac12$. Choose $\theta = \theta(\alpha)$ with $0<\theta<\alpha$ so that
$$
\alpha > \left( 1 - \frac1{M_*} \right) \frac{1+\theta}{1-\theta}.
$$
Choose an integer $p$ as in Proposition \ref{prop:2}. Since both $M_*$ and $\theta$ here have been chosen to depend only on $\alpha$, we have $p = p(\alpha)$.

We next appeal to the two parameter Ramsey's theorem for $3$-uniform hypergraphs, and choose an integer $m$ so that if the collection of all triples in $\{1,\ldots,m\}$ is $2$-colored, then either there exists a subset $A \subset \{1,\ldots,m\}$ of cardinality $K$ so that all triples in $A$ are colored red, or there exists a subset $B \subset \{1,\ldots,m\}$ of cardinality $p$ so that all triples in $B$ are colored blue. By \eqref{eq:CFS}, 
$$
m \le m(\alpha,K) = \exp(c K^{p(\alpha)-2} \log K)
$$
for some absolute constant $c$.

By Proposition \ref{prop:1}, choose constants $\lambda_1$ and $C$ so that the stated conclusion holds true. In the statement of that proposition, $\lambda_1$ and $C$ depend on $\theta$ and $m$; tracing the dependence above we see that in the context of this proof we have $\lambda_1$ and $C$ eventually depend on $\alpha$ and $K$. We impose the second assumption
\begin{equation}\label{eq:lambda-second}
\lambda_0 \le \lambda_1(\theta(\alpha),m(\alpha,K))
\end{equation}
and we assume that a choice of $\lambda_0$ has been made so that both \eqref{eq:lambda-first} and \eqref{eq:lambda-second} are satisfied.

Assume now that $(\gamma,d)$ is a rough $\lambda$-self-expanding finite ordered set of cardinality $n$ so that $L(\gamma) > C D(\gamma)$. Proposition \ref{prop:1} implies that $n \ge m$ and we may choose a subset $F = \{x_{i_1}< \cdots < x_{i_m}\} \subset \gamma$ of cardinality $m$ so that $(F,d)$ satisfies the medial ordered $\SRA(\theta)$ condition. We identify $F$ with the index set $\{1,\ldots,m\}$ and we color the set of all triples $1\le i<j<k\le m$ as follows:
\begin{itemize}
\item Triple $i<j<k$ is colored red if $d(x_j,x_k) \le d(x_i,x_k) + \alpha d(x_i,x_j)$.
\item Triple $i<j<k$ is colored blue if $d(x_j,x_k) \ge d(x_i,x_k) + \alpha d(x_i,x_j)$.
\end{itemize}
Since $\gamma$ is rough $\lambda$-self-expanding and $\lambda < \lambda_0$, $(\gamma,d)$ is $M_*$-bounded turning. The Ramsey-type conclusion above tells us that one of the following statements is valid:
\begin{itemize}
\item[(i)] There exists a subset of $F$ of cardinality $K$ in which all triples are colored red.
\item[(ii)] There exists a subset of $F$ of cardinality $p$ in which all triples are colored blue.
\end{itemize}
Observe that if (ii) is satisfied, then we have found a $p$ element subset $\{x_{j_1},\ldots,x_{j_p}\}$ of $F$ which is a bounded turning set, satisfies the medial ordered $\SRA(\theta)$ condition, and for which the inequality $d(x_{j+1},x_p) \ge d(x_j,x_p) + \alpha d(x_j,x_{j+1})$ holds for all $1\le j < p$. This contradicts the conclusion which we obtain from Proposition \ref{prop:2}, namely, that for any subset of $\gamma$ of cardinality $p$ there exists an index $i$ so that \eqref{eq:zzz} holds.

Hence (i) must be satisfied. Let $A \subset F$ with $\#A = k$ have the property that all triples in $A$ are colored red. Then
\begin{itemize}
\item $(A,d)$ is rough $\alpha$-self-expanding (since $(\gamma,d)$ is rough $\lambda$-self-expanding and $\lambda < \alpha$),
\item $(A,d)$ satisfies the medial ordered $\SRA(\alpha)$ condition (since $\theta<\alpha$), and
\item $(A,d)$ is rough $\alpha$-self-contracting (since all triples in $A$ are colored red).
\end{itemize}
By Lemma \ref{lem:elementary}, $(A,d)$ satisfies the $\SRA(\alpha)$ condition.
\end{proof}

It remains to verify the two technical propositions.

\begin{proof}[Proof of Proposition \ref{prop:2}]
Fix $\theta$, $M$, and $\alpha$ as in the statement. Assume that $(Z,d)$ is a finite, ordered, $M$-bounded turning set satisfying the medial ordered $\SRA(\theta)$ condition, and write $Z = \{ z_1 < \cdots < z_p \}$. We argue by contradiction, so assume that for every choice of $i$, $1\le i < p-1$, the inequality
\begin{equation}\label{eq:2-1}
d(z_{i+1},z_p) \ge d(z_i,z_p) + \alpha d(z_i,z_{i+1})
\end{equation}
holds true. We will find an upper bound for $p$ depending on the data $\theta,M,\alpha$.

The $M$-bounded turning condition implies that $d(z_{p-1},z_p) \le M d(z_1,z_p)$. By induction, and using the medial ordered $\SRA(\theta)$ condition, we deduce that
\begin{equation}\label{eq:2-2}
d(z_i,z_{i+1}) \ge \left( \frac{1-\theta}2 \right)^{p-i-1} d(z_{p-1},z_p) \quad \forall \, 1 \le i < p.
\end{equation}
Summing \eqref{eq:2-1} over this range of indexes gives
$$
d(z_1,z_p) + \alpha \sum_{i=1}^{p-2} d(z_i,z_{i+1}) \le d(z_{p-1},z_p)
$$
and hence
$$
d(z_1,z_p) \le d(z_{p-1},z_p) \left( 1 - \alpha \sum_{i=1}^{p-2} \left( \frac{1-\theta}2 \right)^{p-1-i} \right).
$$
Thus
$$
\alpha \sum_{i=1}^{p-2} \left( \frac{1-\theta}2 \right)^{p-1-i} \le 1 - \frac1M
$$
and so
$$
1 - \left( \frac{1-\theta}2 \right)^{p-2} \le \frac{1+\theta}{1-\theta} \, \left( 1 - \frac1M \right) \, \frac1\alpha < 1.
$$
We conclude that
$$
p \le 2 + \frac{-\log\bigl(1-\tfrac{1+\theta}{1-\theta} \, \left( 1 - \tfrac1M \right) \, \tfrac1\alpha \bigr)}{\log(2/(1-\theta))},
$$
where the upper bound depends only on $\theta$, $M$, and $\alpha$.
\end{proof}

Finally, we turn to the proof of Proposition \ref{prop:1}. This proof is long and technical, and we will break the argument into intermediate steps which will be formulated as individual lemmas. The overall structure of the proof is by contradiction. Assuming that the conclusion does not hold, we define inductively a sequence of $m$-point subsets of $\{1,\ldots,n\}$ for which the $m$-tuples of indexing integers are coordinate-wise non-decreasing. By assumption, the stated conclusion does not hold at any stage of this construction. We continue until the final ($m$th) entry of the sequence is sufficiently close to $n$. Along the way we introduce a suitable weighted sum of entries, which we show to be roughly non-decreasing with respect to the inductive parameter. Once the process stops, we compare the resulting quantity with the discrete length $L(S)$ and diameter $D(S)$ of the original set, and observe that the estimate which we obtain is incompatible with the initial assumption that $L(S) > C D(S)$, provided $C=C(\lambda,m)$ is sufficiently large. This yields the desired contradiction, and thereby completes the proof of the proposition. 

\begin{proof}[Proof of Proposition \ref{prop:1}]
We will prove the result by contradiction, and the necessary constants $\lambda_1(\theta,m)$ and $C(\theta,m)$ will be determined at the conclusion of the proof, see \eqref{eq:lambda-nought} and \eqref{eq:C}. Let $S=\{x_1<\cdots<x_n\}$ be a finite, ordered, rough $\lambda$-self-expanding set with $L(S) > C D(S)$ and $\lambda \le \lambda_1$. We first show that if the constant $C$ is chosen sufficiently large (in terms of $m$ and $\theta$), then $n \ge m$. Indeed, using the $M_{\lambda_1}$-bounded turning property of $S$, we find that $d(x_i,x_{i+1}) \le M_{\lambda_1} d(x_1,x_n)$ for each $i=0,\ldots,n-1$. Hence
$$
C d(x_1,x_n) = C D(S) < L(S) = \sum_i d(x_i,x_{i+1}) \le M_{\lambda_1} (n-1) d(x_1,x_n).
$$
Provided we choose 
\begin{equation}\label{eq:C0}
C \ge M_{\lambda_1} (m-1),
\end{equation}
it follows that $n \ge m$ as desired. Note that the actual choice of $C(\theta,m)$ which we make in \eqref{eq:C} satisfies \eqref{eq:C0}. Moreover, the value for $\lambda_1$ which we determine in \eqref{eq:lambda-nought} satisfies $\lambda_1 \le \tfrac12$, whence
\begin{equation}\label{eq:Mlambda1}
M_{\lambda_1} \le 6.
\end{equation}

We now start the proof by contradiction, and we suppose that every $m$-element subset $E \subset S$ fails to satisfy the medial ordered $\SRA(\theta)$ condition. For each such subset $E$, choose indexes $i_E<j_E<k_E$ from the set $\{1,\ldots,n\}$ so that
\begin{equation}\label{eq:contradiction-hypothesis}
d(x_{i_E},x_{k_E}) > d(x_{i_E},x_{j_E}) + \theta \, d(x_{j_E},x_{k_E}).
\end{equation}
We define inductively a sequence of $m$-element subsets $P^1,P^2,\ldots,P^T$ of $\{1,\ldots,n\}$. For each $t$, let $E^t$ be the corresponding subset of $S$, i.e., $E^t = \{x_{i}:i \in P^t\}$.

\begin{remark}\label{rem:rem1}
The number $T$ of subsets in the sequence which we will define is on the order of $n/m$, see Remark \ref{rem:rem2}. As a result, we need to make sure as the proof continues that our eventual choice of the comparison constant $C$ does {\bf not} depend on $T$.
\end{remark}

First, we set
$$
P^1 = \{1,\ldots,m\}.
$$
Next, assume that $P^t$ has been defined, with $P^t = \{p_1^t,\ldots,p_m^t\} \subset \{1,\ldots,n\}$. Choose indexes
$$
i^t:= i_{P^t} < j^t:=j_{P^t} < k^t:=k_{P^t}
$$
satisfying
$$
d(x_{i^t},x_{k^t}) > d(x_{i^t},x_{j^t}) + \theta \, d(x_{j^t},x_{k^t}).
$$
The successor set $P^{t+1}$ will be defined by removing all elements of $P^t$ between $j^t$ and $k^t-1$ and adding in an equal number of elements immediately to the right of $p_m^t$. More precisely, let
$$
d_t:= \# ( P^t \cap \{j^t,j^t+1,\ldots,k^t-1\} ).
$$
We consider two cases.
\begin{itemize}
\item[Case 1:] $d_t+p^t_m> n$. In this case we stop the inductive process, and we set $T = t$.
\item[Case 2:] $d_t+p^t_m \le n$. In this case, we set
$$
P^{t+1} = P^t \setminus \{j^t,j^t+1,\ldots,k^t-1\} \cup \{p_m^t +1,\ldots,p_m^t + d_t \}.
$$
\end{itemize}
See Figure \ref{fig:linept} for a graphical visualization of this process.

\begin{figure}[h]
\begin{center}
\includegraphics[width=12cm]{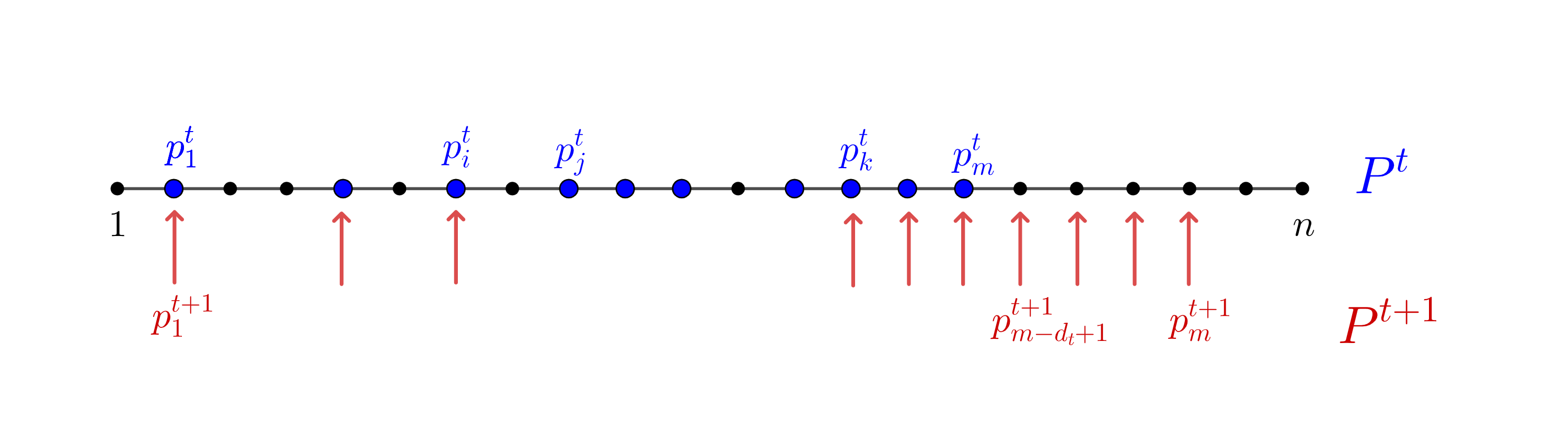}
\caption{Distribution of elements in $P^t$ and $P^{t+1}$}\label{fig:linept}
\end{center}
\end{figure}

\begin{remark}\label{rem:rem2}
It is clear from the construction that the sequence $(p_k^t:t=1,2,\ldots)$ is nondecreasing for each $k=1,\ldots,m$. Moreover, for all choices of $t$ we have $p_m^{t+1} = p_m^t + d_t$, whence
$$
p_m^T = m + \sum_{s=1}^{T-1} d_s.
$$
Since $1 \le d_t \le m-1$ for all $t$ we conclude that $p_m^T \le mT+(T-1)$. On the other hand, the termination condition (Case 1) is $p_m^T > n-d_T \ge n-m$. Hence $n-m\le mT+T-1$ or $T \ge \tfrac{n-m+1}{m+1}$.
\end{remark}

Next, we fix the weighting factor
\begin{equation}\label{eq:rho}
\rho := 6 \, \frac{m(m+1)}{2\theta} > 1,
\end{equation}
where the coefficient $6$ has been determined by the bound in \eqref{eq:Mlambda1}.

For each $t=1,\ldots,T$, define the following weighted sum:
$$
\cS^t := \sum_{\substack{(a,b) \\ 1 \le a<b \le m}} \rho^{m-1-a} d(x_{p_a^t},x_{p_b^t}).
$$
The sum in question is taken over two parameters $a$ and $b$ satisfying $1\le a<b\le m$. In what follows, we omit the explicit reference to the pair $(a,b)$ in later instances of such sums.

\begin{lemma}\label{lem:prop1lemma}
For each $t$,
\begin{equation}\label{eq:prop1lemma}
\cS^t \le \cS^{t+1} - \sum_{u=p_m^t}^{p_m^{t+1}-1} d(x_u,x_{u+1}) + C_1(m,\theta) \, \lambda \, \sum_{b=1}^m d(x_{p_b^t},x_{p_b^{t+1}})
\end{equation}
for some constant $C_1(m,\theta)$.
\end{lemma}

Assuming temporarily that Lemma \ref{lem:prop1lemma} is true, we complete the proof of the proposition. Summing \eqref{eq:prop1lemma} from $t=1$ to $t=T-1$ yields
$$
\cS^1 \le \cS^T - \sum_{u=m}^{p_m^T-1} d(x_u,x_{u+1}) + C_1(m,\theta) \, \lambda \, \sum_{t=1}^{T-1} \sum_{b=1}^m d(x_{p_b^t},x_{p_b^{t+1}}).
$$
Considering only terms with $b=a+1$ and dropping all powers of the weighting factor $\rho$, we have
$$
\cS^1 = \sum_{1\le a<b\le m} \rho^{m-1-a} d(x_a,x_b) \ge \sum_{i=1}^{m-1} d(x_i,x_{i+1}).
$$
Thus
$$
\sum_{i=1}^{m-1} d(x_i,x_{i+1}) + \sum_{u=m}^{p_m^T-1} d(x_u,x_{u+1}) \le \cS^T + C_1(m,\theta) \, \lambda \, \sum_{b=1}^m \left( \sum_{t=1}^{T-1} d(x_{p_b^t},x_{p_b^{t+1}}) \right),
$$
or equivalently,
$$
L(S) = \sum_{i=1}^n d(x_i,x_{i+1}) \le \cS^T + \sum_{u=p_m^T}^{n-1} d(x_u,x_{u+1}) + C_1(m,\theta) \, \lambda \, \sum_{b=1}^m \left( \sum_{t=1}^{T-1} d(x_{p_b^t},x_{p_b^{t+1}}) \right).
$$
By Lemma \ref{Lemma:expcurveBT}, $(S,d)$ is $M_{\lambda_1}$-bounded turning. Hence $d(x_u,x_{u+1}) \le M_{\lambda_1} d(x_1,x_n) = M_{\lambda_1} D(S)$ for all $p_m^T \le u \le {n-1}$, and so
$$
L(S) \le \cS^T + m \, M_{\lambda_1} \, D(S) + C_1(m,\theta) \, \lambda \, \sum_{b=1}^m \left( \sum_{t=1}^{T-1} d(x_{p_b^t},x_{p_b^{t+1}}) \right).
$$
For each such $b$ and $t$, we have $d(x_{p_b^t},x_{p_b^{t+1}}) \le d(x_{p_b^t},x_{p_b^t+1}) + \cdots + d(x_{p_b^{t+1}-1},x_{p_b^{t+1}})$ by the triangle inequality. Thus
$$
\sum_{b=1}^m \sum_{t=1}^{T-1} d(x_{p_b^t},x_{p_b^{t+1}}) \le \sum_{b=1}^m \sum_{u=1}^{n-1} d(x_u,x_{u+1}) = m \, L(S)
$$
and we conclude that
$$
L(S) \le \cS^T + m \, M_{\lambda_1} \, D(S) + m \, C_1(m,\theta) \, \lambda L(S).
$$
Finally, we estimate the term $\cS^T$. Using the $M_{\lambda_1}$ bounded turning property again, we find
\begin{equation*}\begin{split}
\cS^T &= \sum_{1 \le a < b \le m} \rho^{m-1-a} \, d(x_{p_a^T},x_{p_b^T}) \\
&\le M_{\lambda_1} \, \sum_{1 \le a < b \le m} \rho^{m-1-a} \, d(x_1,x_n).
\end{split}\end{equation*}
The quantity $\sum_{1 \le a < b \le m} \rho^{m-1-a}$ can be explicitly computed as a function of $m$ and $\rho$, alternatively, we can use the trivial upper bound
\begin{equation}\label{eq:sum-a-b-rho-power}
\sum_{1 \le a < b \le m} \rho^{m-1-a} \le \frac{m(m-1)}{2} \rho^{m-2}.
\end{equation}
Using \eqref{eq:sum-a-b-rho-power} for simplicity, we obtain 
$$
\cS^T \le \frac{m(m-1)}{2} \, M_{\lambda_1} \, \rho^{m-2} D(S)
$$
and therefore
\begin{equation}\label{eq:L-to-D-upper-bound}
L(S) \le \frac{\bigl(\tfrac{m(m-1)}{2}\rho^{m-2} + m\bigr)M_{\lambda_1}}{1-m \, C_1(m,\theta) \, \lambda} \, D(S) \le \frac{6\bigl(\tfrac{m(m-1)}{2}\rho^{m-2} + m\bigr)}{1-m \, C_1(m,\theta) \, \lambda} \, D(S)
\end{equation}
provided $\lambda<(m C_1(m,\theta))^{-1}$. (Here we again used \eqref{eq:Mlambda1}). The bound in \eqref{eq:L-to-D-upper-bound} tells us how to choose the constants $\lambda_1$ and $C$ in the statement of Proposition \ref{prop:1}. To wit, set
\begin{equation}\label{eq:lambda-nought}
\lambda_1(\theta,m) := \frac{1}{2m C_1(m,\theta)}
\end{equation}
and
\begin{equation}\label{eq:C}
C(\theta,m) := 12\left( \frac{m(m-1)}{2}\rho^{m-2} + m \right).
\end{equation}
\end{proof}

\begin{proof}[Proof of Lemma \ref{lem:prop1lemma}]
We break $\cS^t$ into three separate terms, as follows:
\begin{equation}\begin{split}\label{eq:S-t-three-terms}
&\sum_{\substack{1 \le a < b \le m \\ p_a^t < j^t}} \rho^{m-1-a} d(x_{p_a^t},x_{p_b^t}) + \sum_{\substack{1 \le a < b \le m \\ p_a^t = j^t}} \rho^{m-1-a} d(x_{p_a^t},x_{p_b^t}) + \sum_{\substack{1 \le a < b \le m \\ p_a^t > j^t}} \rho^{m-1-a} d(x_{p_a^t},x_{p_b^t}) \\
& \qquad =: \bI^t + \bI\bI^t + \bI\bI\bI^t.
\end{split}\end{equation}

\begin{lemma}\label{lem:prop1lemmaI}
If the index $\hat{a}$ is chosen so that $p_{\hat{a}}^t = i^t$, then 
$$
\bI^t \le \bI^{t+1} - \theta \rho^{m-1-\hat{a}} d(x_{j^t},x_{k^t}) + C_1(m,\theta) \lambda \, \sum_{b=1}^m d(x_{p_b^t},x_{p_b^{t+1}}).
$$
\end{lemma}

\begin{lemma}\label{lem:prop1lemmaII}
If the index $\hat{\hat{a}}$ is chosen so that $p_{\hat{\hat{a}}}^t = j^t$, then
$$
\bI\bI^t \le \bI\bI^{t+1} - \rho^{m-1-\hat{\hat{a}}} d(x_{k^t},x_{p_m^t+1}) + M_\lambda (m- \hat{\hat{a}}) \, \rho^{m-1-\hat{\hat{a}}} d(x_{j^t},x_{k^t}).
$$
\end{lemma}

\begin{lemma}\label{lem:prop1lemmaIII}
\begin{equation*}\begin{split}
\bI\bI\bI^t 
& \le \bI\bI\bI^{t+1} - \sum_{\substack{1 \le a<b \le m \\ p_a^t > j^t}} \rho^{m-1-a} d(x_{p_a^{t+1}},x_{p_b^{t+1}}) \\
& \qquad + M_\lambda \, \sum_{\substack{1\le a<b\le m \\ p_a^t > j^t}} \rho^{m-1-a} d(x_{j^t},x_{k^t}) + M_\lambda \, \sum_{\substack{1 \le a<b\le m \\ p_a^t > j^t, p_b^t > k^t}} \rho^{m-1-a} d(x_{k^t},x_{p_m^t+1}).
\end{split}\end{equation*}
\end{lemma}

Postponing the proofs of these three lemmas, we complete the proof of Lemma \ref{lem:prop1lemma}. Summing the estimates for $\bI^t$, $\bI\bI^t$, and $\bI\bI\bI^t$ from the three lemmas, we obtain
\begin{equation}\label{eq:St-induction}
\cS^t \le \cS^{t+1} - \cR_1 - \cR_2 - \cR_3 + C_1(m,\theta) \lambda \, \sum_{b=1}^m d(x_{p_b^t},x_{p_b^{t+1}}),
\end{equation}
where
$$
\cR_1 = \left( \theta \rho^{m-1-\hat{a}} - M_\lambda \, \sum_{\substack{1\le a<b\le m \\ p_a^t > j^t}} \rho^{m-1-a} - M_\lambda (m- \hat{\hat{a}}) \, \rho^{m-1-\hat{\hat{a}}} \right) \, d(x_{j^t},x_{k^t}),
$$
$$
\cR_2 = \left( \rho^{m-1-\hat{\hat{a}}} - M_\lambda \, \sum_{\substack{1 \le a<b\le m \\ p_a^t > j^t, p_b^t > k^t}} \rho^{m-1-a} \right) d(x_{k^t},x_{p_m^t+1}),
$$
and
$$
\cR_3 = \sum_{\substack{1 \le a<b \le m \\ p_a^t > j^t}} \rho^{m-1-a} d(x_{p_a^{t+1}},x_{p_b^{t+1}}).
$$
Recall that the indexes $\hat{a}$ and $\hat{\hat{a}}$ were chosen so that $p_{\hat{a}}^t = i^t$ and $p_{\hat{\hat{a}}}^t = j^t$.

We now show that $\cR_1 \ge 0$ and one of the following two conclusions holds true:
\begin{itemize}
\item $\cR_2 \ge 0$ and $\cR_3 \ge  \sum_{u=p_m^t}^{p_m^{t+1}-1} d(x_u,x_{u+1})$, or
\item $\cR_2 \ge d(x_{p_m^t},x_{p_m^{t+1}})$ and $\cR_3 \ge \sum_{u=p_m^t+1}^{p_m^{t+1}-1} d(x_u,x_{u+1})$.
\end{itemize}
Using these bounds in \eqref{eq:St-induction} we derive the desired conclusion
$$
\cS^t \le \cS^{t+1} - \sum_{u=p_m^t}^{p_m^{t+1}} d(x_u,x_{u+1}) + C_1(m,\theta) \lambda \, \sum_{b=1}^m d(x_{p_b^t},x_{p_b^{t+1}})
$$
and thus complete the proof of Lemma \ref{lem:prop1lemma}.

We first analyze the remainder term $\cR_1$. Since $p_{\hat{a}}^t = i^t$, it follows that $a \ge \hat{a}+1$ whenever $p_a^t \ge j^t$. In particular, since $p_{\hat{\hat{a}}}^t = j^t$ we have $\hat{\hat{a}} \ge \hat{a} + 1$. It follows that
$$
\sum_{\substack{1\le a<b\le m \\ p_a^t > j^t}} \rho^{m-1-a} \le \frac{m(m-1)}{2} \rho^{m-2-\hat{a}}
$$
and
$$
(m- \hat{\hat{a}}) \, \rho^{m-1-\hat{\hat{a}}} \le m \, \rho^{m-2-\hat{a}},
$$
and hence
$$
\cR_1 \ge \left( \theta  \rho^{m-1-\hat{a}} - \frac{m(m+1)}{2} M_\lambda \rho^{m-2-\hat{a}} \right) d(x_{j^t},x_{k^t}).
$$
In view of the choice of $\rho$ in \eqref{eq:rho}, we conclude that $\cR_1 \ge 0$.

We analyze the remainder terms $\cR_2$ and $\cR_3$ simultaneously, dividing into two cases:
\begin{itemize}
\item[(i)] $k^t < p_m^t$, and 
\item[(ii)] $k^t = p_m^t$.
\end{itemize}

We start with case (i). Note that by construction, if $j^t < p_a^t < p_b^t \le p_m^t$, then $k^t < p_a^{t+1} < p_b^{t+1} \le p_m^{t+1}$. Moreover, each of the pairs $(u,u+1)$ with $p_m^t \le u < p_m^{t+1}-1$ occurs as a pair of the form $(p_a^{t+1},p_b^{t+1})$ for some $a<b$ with $p_a^t > j^t$. Thus
$$
\cR_3 = \sum_{\substack{1 \le a<b \le m \\ p_a^t > j^t}} \rho^{m-1-a} d(x_{p_a^{t+1}},x_{p_b^{t+1}}) \ge \sum_{u=p_m^t}^{p_m^{t+1}-1} d(x_u,x_{u+1}).
$$
Furthermore, $p_{\hat{\hat{a}}}^t = j^t$, and so $a \ge \hat{\hat{a}} + 1$ whenever $p_a^t > j^t$. Thus
$$
\sum_{\substack{1 \le a<b\le m \\ p_a^t > j^t, p_b^t > k^t}} \rho^{m-1-a} d(x_{k^t},x_{p_m^t+1}) \le \frac{m(m-1)}{2} \rho^{m-2-\hat{\hat{a}}}
$$
and hence
$$
\cR_2 \ge \left( \rho^{m-1-\hat{\hat{a}}} - \frac{m(m-1)}{2} M_\lambda \rho^{m-2-\hat{\hat{a}}} \right) d(x_{k^t},x_{p_m^t+1}).
$$
We conclude that $\cR_2 \ge 0$ if
\begin{equation}\label{eq:rho2}
\rho \ge \frac{m(m-1)}{2} M_\lambda
\end{equation}
but note that \eqref{eq:rho2} is satisfied if $\rho$ is chosen as in \eqref{eq:rho}, since $\lambda \le \lambda_1 \le \tfrac12$.

Alternatively, assume that case (ii) holds. In this case,
$$
\cR_2 = \rho^{m-1-\hat{\hat{a}}} d(x_k^t,x_{p_m^t+1}) \ge d(x_{p_m^t},x_{p_m^t+1}).
$$
Also, if $j^t < p_a^t < p_b^t \le p_m^t = k^t$, then $k^t < p_a^{t+1} <  p_b^{t+1}$, and each of the pairs $(u,u+1)$ with $p_m^t+1 \le u < p_m^{t+1}-1$ occurs as a pair of the form $(p_a^t,p_b^t)$ for some $a<b$ with $p_a^t>j^t$. Thus
$$
\cR_3 = \sum_{\substack{1 \le a<b \le m \\ p_a^t > j^t}} \rho^{m-1-a} d(x_{p_a^{t+1}},x_{p_b^{t+1}}) \ge \sum_{u=p_m^t+1}^{p_m^{t+1}-1} d(x_u,x_{u+1}).
$$
This finishes the proof of Lemma \ref{lem:prop1lemma}.
\end{proof}

Finally, we prove Lemmas \ref{lem:prop1lemmaI}, \ref{lem:prop1lemmaII}, and \ref{lem:prop1lemmaIII}.

\begin{proof}[Proof of Lemma \ref{lem:prop1lemmaI}]
Recall that we fix $\hat{a}$ so that $p_{\hat{a}}^t = i^t$, and we consider the sum of terms in $\cS^t$ over pairs $(a,b)$ so that $p_a^t < j^t$. For such $a$ and $b$, we necessarily have $p_a^{t+1} = p_a^t$ and $p_b^{t+1} \ge p_b^t$ by construction. If we also fix $\hat{b}$ so that $p_{\hat{b}}^t = j^t$, then $p_{\hat{b}}^{t+1} = k^t$.

For each pair $a<b$ so that $(a,b) \ne (\hat{a},\hat{b})$, the rough $\lambda$-self-expanding condition implies that
\begin{equation}\label{eq:prop1lemmaI-1}
d(x_{p_a^t},x_{p_b^t}) \le d(x_{p_a^{t+1}},x_{p_b^{t+1}}) + \lambda d(x_{p_b^t},x_{p_b^{t+1}}).
\end{equation}
For the particular choice $(a,b) = (\hat{a},\hat{b})$, we have the following alternate inequality which follows from \eqref{eq:contradiction-hypothesis}:
\begin{equation}\label{eq:prop1lemmaI-2}
d(x_{p_{\hat{a}}^t},x_{p_{\hat{b}}^t}) < d(x_{p_{\hat{a}}^{t+1}},x_{p_{\hat{b}}^{t+1}}) - \theta d(x_{j^t},x_{k^{t}}).
\end{equation}
Summing over all $a<b$ for which $p_a^t<j^t$, we conclude that
$$
\bI^t \le \bI^{t+1} - \theta \rho^{m-1-\hat{a}} d(x_{j^t},x_{k^t}) + \lambda \sum_{\substack{1\le a<b\le m \\ p_a^t < j^t, (a,b) \ne (\hat{a},\hat{b})}} \rho^{m-1-a} d(x_{p_b^t},x_{p_b^{t+1}}).
$$
To complete the proof of this lemma, we estimate
$$
 \sum_{\substack{1\le a<b\le m \\ p_a^t < j^t, (a,b) \ne (\hat{a},\hat{b})}} \rho^{m-1-a} d(x_{p_b^t},x_{p_b^{t+1}}) \le \sum_{a=1}^{m-1} \rho^{m-1-a} \sum_{b=2}^m d(x_{p_b^t},x_{p_b^{t+1}}) \le C_1(m,\theta) \sum_{b=2}^m d(x_{p_b^t},x_{p_b^{t+1}}),
$$
where
$$
C_1(m,\theta) = \frac{\rho^{m-1}-1}{\rho-1}
$$
and $\rho$ is chosen as in \eqref{eq:rho}. \end{proof}

\begin{proof}[Proof of Lemma \ref{lem:prop1lemmaII}]
Since we chose $\hat{\hat{a}}$ so that $p_{\hat{\hat{a}}}^t = j^t$, in this case we restrict to the sum of terms in $\cS^t$ over pairs $(a,b)$ with $a = \hat{\hat{a}}$. Our goal is to estimate
$$
\rho^{m-1-\hat{\hat{a}}} \sum_{\hat{\hat{a}} < b < m} d(x_{p_{\hat{\hat{a}}}^t},x_{p_b^t}).
$$
In this case, we have $p_{\hat{\hat{a}}}^{t+1} = k^t$ by construction. Moreover,
\begin{equation}\label{eq:inclusion}
\{ u \in P^t \, : \, u>k^t \} \cup \{ p_m^t+1 \} \subset \{ u \in P^{t+1} \, : \, u>k^t \}.
\end{equation}
For $b$ with $p_b^t > k^t$ we use the triangle inequality to get
\begin{equation}\label{eq:prop1lemmaII-1}
d(x_{p_{\hat{\hat{a}}}^t},x_{p_b^t}) \le d(x_{j^t},x_{k^t}) + d(x_{k^t},x_{p_b^t}),
\end{equation}
while for $b$ with $p_b^t \le k^t$ we use the $M_\lambda$-bounded turning property to get
\begin{equation}\label{eq:prop1lemmaII-2}
d(x_{p_{\hat{\hat{a}}}^t},x_{p_b^t}) \le M_\lambda \, d(x_{j^t},x_{k^t}).
\end{equation}
Thus
$$
\sum_{\hat{\hat{a}} < b \le m} d(x_{p_{\hat{\hat{a}}}^t},x_{p_b^t}) \le M_\lambda \sum_{\hat{\hat{a}} < b \le m} d(x_{j^t},x_{k^t}) + 
\sum_{\substack{\hat{\hat{a}} < b \le m \\ p_b^t > k^t}} d(x_{k^t},x_{p_b^t}).
$$
Due to \eqref{eq:inclusion},
$$
\sum_{\substack{\hat{\hat{a}} < b \le m \\ p_b^t > k^t}}  d(x_{p_{\hat{\hat{a}}}^t},x_{p_b^t}) \le
\sum_{\hat{\hat{a}} < b \le m} d(x_{p_{\hat{\hat{a}}}^{t+1}},x_{p_b^{t+1}}) - d(x_{k^t},x_{p_m^t+1}).
$$
Combining all of the above yields
$$
\bI\bI^t \le \bI\bI^{t+1} - \rho^{m-1-\hat{\hat{a}}} d(x_{k^t},x_{p_m^t+1}) + M_\lambda (m- \hat{\hat{a}}) \, \rho^{m-1-\hat{\hat{a}}} d(x_{j^t},x_{k^t})
$$
as desired.
\end{proof}

\begin{proof}[Proof of Lemma \ref{lem:prop1lemmaIII}]
In this final case, we consider pairs $a<b$ with $p_a^t > j^t$. If $p_b^t \le k^t$ then $j^t < p_a^t < p_b^t \le k^t$ and so
$$
d(x_{p_a^t},x_{p_b^t}) \le M_\lambda \, d(x_{j^t},x_{k^t})
$$
by the bounded turning condition. On the other hand, if $k^t < p_b^t$ then
$$
d(x_{p_a^t},x_{p_b^t}) \le d(x_{p_a^t},x_{k^t}) + d(x_{k^t},x_{p_b^t}) \le M_\lambda \, \left( d(x_{j^t},x_{k^t}) + d(x_{k^t},x_{p_m^t+1}) \right).
$$
It follows that
\begin{equation*}\begin{split}
\bI\bI\bI^t 
& \le \bI\bI\bI^{t+1} - \sum_{\substack{1 \le a<b \le m \\ p_a^t > j^t}} \rho^{m-1-a} d(x_{p_a^{t+1}},x_{p_b^{t+1}}) \\
& \qquad + M_\lambda \, \sum_{\substack{1\le a<b\le m \\ p_a^t > j^t}} \rho^{m-1-a} d(x_{j^t},x_{k^t}) + M_\lambda \, \sum_{\substack{1 \le a<b\le m \\ p_a^t > j^t, p_b^t > k^t}} \rho^{m-1-a} d(x_{k^t},x_{p_m^t+1})
\end{split}\end{equation*}
as desired.
\end{proof}

\section{Questions and further comments}\label{sec:questions}

In this section we collect several open problems motivated by our work, along with some relevant discussion and further comments.

\begin{question}\label{Q:optimal-lambda}
For a given metric space $(X,d)$, compute
\begin{equation}\label{eq:lambda0star}
\lambda_0^*(X) := \sup \left\{ \lambda_0 \in (-1,1] : {\mbox{every bounded, rough $\lambda$-self monotone curve} \atop \mbox{in $(X,d)$ with $\lambda < \lambda_0$ is rectifiable} } \right\}. 
\end{equation}
If no such value $\lambda_0$ exists, we set $\lambda_0^*(X) = -1$. Recall that $(-1)$-self monotone curves in any metric space are just geodesics, and hence are always rectifiable.
\end{question}

By results in \cite{DaDuDe}, we have $\lambda_0^*(\R^n) \ge \frac1n$ for every $n \ge 2$. Our main theorem (Theorem \ref{thm:rectifiability}) states that $\lambda_0^*(X) > 0$ whenever $(X,d)$ is $\SRA(\alpha)$ free for some $\tfrac12 < \alpha < 1$. By results discussed earlier in this paper, we have that $\lambda_0^*(X) \le 0$ when $X$ is the sub-Riemannian Heisenberg group, or Hilbert space, or Laakso space, or the metric graphs considered in Example \ref{example:metrictree}.

We do not know any examples of $\SRA(\alpha)$ free metric spaces $X$ for which $\lambda_0^*(X) < 1$, so it could be true that the value in \eqref{eq:lambda0star} for such spaces is always equal to one.

In connection with the previous discussion about the possible $\SRA(0)$ freeness of the Heisenberg group, we pose the following question.

\begin{question}\label{Q:Heis-SRA(0)}
Is the first Heisenberg group $\Heis^1$ equipped with the Carnot--Carath\'eodory metric $\SRA(0)$ free? If so, what is the largest size of an ultrametric subset of $(\Heis^1,d_{cc})$?
\end{question}

In \cite{kp:equilateral}, the {\it equilateral dimension} of a metric space $(X,d)$ is defined to be the cardinality of the largest equilateral subset of $X$, where a subset $A \subset X$ is said to be {\it equilateral} if there exists $c>0$ so that $d(x,y) = c$ for all distinct $x,y$ in $A$. Clearly, every equilateral subset is an ultrametric subset. The equilateral dimension of $\R^n$ in the Euclidean metric is $n+1$, and equilateral dimensions of other finite dimensional normed vector spaces are known. See \cite{kp:equilateral} for a review of these results. In \cite{kp:equilateral}, the authors show that the equilateral dimension of $(\Heis^1,d_H)$ is $4$, where $d_H$ denotes the Kor\'anyi metric on the Heisenberg group $\Heis^1$. However, it is not clear how to extend their argument to study Question \ref{Q:Heis-SRA(0)}.

The study of self-contracting curves in Euclidean space was heavily motivated by applications to gradient flows for convex functions. This connection has been worked out in detail in Euclidean and Riemannian spaces, but to the best of our knowledge the following possible extension has not been considered.

\begin{question}\label{Q:Heis-SC}
Are the integral curves associated to the horizontal gradient flow of an H-convex function in a domain of the Heisenberg group $\Heis^1$ necessarily self-contracting, in either the Kor\'anyi metric $d_H$ or the CC metric $d_{cc}$?
\end{question}

Regarding roughly self-contracting curves, the following question also remains open, even in Euclidean space.

\begin{question}\label{Q:rough-SC}
Do roughly self-contracting curves arise as integral curves associated to a gradient flow? If so, what is the relevant class of functions whose gradient flows generate such curves?
\end{question}

Finally, we collect a few additional questions of a more technical nature, which arise in connection with specific remarks or examples in this paper.

\begin{question}\label{Q:sra-vs-snowflake}
Let $1<p<\infty$ and $\alpha<1$. Does there exist a set $X \subset \ell^p$ so that $(X,||\cdot||_p)$ satisfies the $\SRA(\alpha)$ condition, but $(X,||\cdot||_p^q)$ fails to be a metric space for any $q>1$?
\end{question}

\begin{question}\label{Q:Hilbert}
For which values $\alpha \in (0,1)$ does there exist a set $A$ in $\ell_2$ which is $\SRA(\alpha)$, yet $(A,||\cdot||_2^q)$ fails to be a metric space for any $q>1$?
\end{question}

Recall that in Example \ref{ex:no-converse} we saw that metric spaces $X$ with the above property exist for any $\alpha>0$, however, the examples constructed there did not admit any obvious embedding into any uniformly convex Banach space. For subsets of Hilbert space the best example we currently have is given in section \ref{appendix-b}, and only valid for $1/\sqrt{2} < \alpha < 1$.

\begin{question}\label{Q:main-thm-1}
Does there exist any value $\alpha \le \tfrac12$ for which the conclusion in Theorem \ref{thm:rectifiability} holds true for any $\SRA(\alpha)$ metric space?
\end{question}

\begin{question}\label{Q:main-thm-2}
If a metric space $(X,d)$ contains a bounded, unrectifiable self-monotone curve, must $(X,d)$ be $\SRA(\alpha)$ full for some choice of $0\le \alpha < 1$?
\end{question}

\begin{question}\label{Q:Laakso-existence}
Does the Laakso graph $G_\infty$ contain a bounded, self-contracted, unrectifiable set?
\end{question}

Recall that in Example \ref{example:laakso} we constructed an infinite ultrametric subset $F_\infty$ as a vertical slice of $G_\infty$. As noted before, $F_\infty$ may be equipped with a self-contracted ordering, but we do not know whether such an ordering exists which makes $F_\infty$ unrectifiable. There is, however, a self-contracted rectifiable ordering of $F_\infty$. Such ordering may be constructed as the limit of a consistent sequence of self-contracted orderings of the finite sets $F_m$, with a uniform upper bound on length. The set $F_m$ can be identified with the vertices of the cube $\{0,1\}^m$, and a suitable self-contracted ordering can be obtained from a Hamiltonian cycle on the associated graph. Using the formula \eqref{eq:d-on-f-m} for the metric on $F_m$, one computes that the length of $F_m$ in this ordering is bounded above by $\tfrac23$, uniformly in $m$. Details are left to the reader.

As noted in Remark \ref{rem:Laakso-remark}, the Laakso graph contains larger ultrametric subsets, which in turn can be equipped with self-contracted orderings. However, we do not know whether any such ordering is unrectifiable.

\section{Appendix A: A doubling metric space which is neither $\SRA(\alpha)$-free nor $\SRA(\alpha)$-full}\label{appendix-a}

The following example was provided to the authors by Tuomas Hyt\"onen \cite{hyt:personal}. We are grateful to him for permission to include it in this paper.

Let 
$$
d_\infty(x,y) := \inf \{ 2^{-j} : \exists I \in \cD_j \mbox{ with } x,y \in I \}
$$
be the standard dyadic ultrametric on $[0,1)$, where
$$
\cD_j := \{ 2^{-j} [k,k+1) : k=0,1,\ldots,2^j-1 \}
$$
are the dyadic subintervals of $[0,1)$ of length $2^{-j}$. 

For each $j$, let
$$
d_j(x,y) := \begin{cases} 
\abs{x-y}, & \mbox{if }d_\infty(x,y)\leq 2^{-j}, \\ 
d_\infty(x,y), &\mbox{otherwise}.
\end{cases}
$$
Note that $d_\infty(x,y) \ge |x-y|$, and hence $d_j(x,y) \ge |x-y|$. We claim that each $d_j$ is a metric. It suffices to verify the triangle inequality $d_j(x,y) \le d_j(x,z) + d_j(z,y)$. If $d_\infty(x,y) < 2^{-j}$, then
$$
  d_j(x,y)=\abs{x-y}\leq\abs{x-z}+\abs{z-y}\leq d_j(x,z)+d_j(z,y).
$$
If $d_\infty(x,y)\geq 2^{-j}$, then
\begin{equation*}\begin{split}
d_j(x,y) =d_\infty(x,y) &\leq\max\{d_\infty(x,z),d_\infty(z,y)\} \\
  &\stackrel{(*)}{=} \max\{d_j(x,z),d_j(z,y)\}\leq d_j(x,z)+d_j(z,y),
\end{split}
\end{equation*}
where in $(*)$ we used the fact that the larger of the values $d_\infty(u,z)$, $u\in\{x,y\}$, satisfies $d_\infty(u,z)\geq 2^{-j}$ and hence $d_\infty(u,z)=d_j(u,z)$.

On $X:=[0,1)\times\N$, let
\begin{equation*}
  d((x,m),(y,n)):=\begin{cases} d_n(x,y), & \mbox{if }m=n, \\ \abs{m-n}, & \mbox{if }m\neq n,\end{cases}
\end{equation*}
and
\begin{equation*}
  \mu(E):=\sum_{n\in\N}\abs{E_n},\qquad E=\bigcup_{n\in\N}E_n\times\{n\},
\end{equation*}
where $\abs{E_n}$ is the Lebesgue measure. 

We next claim that $(X,d)$ is a doubling metric space. For the proof, we introduce the notation $I_j(x)$, for $x \in [0,1)$ and $j \in \N$, to denote the unique element of $\cD_j$ which contains $x$. With this notation in place, and with a reminder that we denote by $B_d(p,r)$ the closed ball in a metric space $(X,d)$ with center $p$ and radius $r$, we invite the reader to verify that
$$
B_d((x,m),r) = \begin{cases}  
\{ y \in I_m(x) \, : \, |x-y| \le r \} \times \{m\} & \mbox{if $r \le 2^{-m}$,} \\
I_j(x) \times \{m\} & \mbox{if $2^{-j} < r \le 2^{1-j}$ and $j\in\{2\ldots,m\}$,} \\
I_1(x) \times \{m\} & \mbox{if $2^{-1} < r < 1$, and} \\
[0,1) \times \{n \in \N : |m-n| \le r \} & \mbox{if $r \ge 1$.} \end{cases}
$$
It follows that there exist constants $0<c_1 \le c_2 < \infty$ so that $c_1 r \le \mu(B_d(x,m),r) \le c_2 r$ for all $(x,m) \in X$ and all $r>0$. Hence $\mu$ is an Ahlfors $1$-regular measure on $(X,d)$, which implies that $(X,d)$ is doubling.

The metric space $(X,d)$ contains arbitrarily large ultrametric subsets, since for each $j$ the set $\{2^{-j}k: k=0,1,\ldots,2^j-1\}\times\{j\}$ is an ultrametric set of size $2^j$. Thus $(X,d)$ is not $\SRA(\alpha)$-free for any $0\le \alpha<1$. On the other hand, for any such $\alpha$, all $\SRA(\alpha)$ subsets of $X$ are finite. Indeed, suppose that $Y$ is an infinite $\SRA(\alpha)$ subset of $X$, and write
\begin{equation*}
  Y=\bigcup_{n\in\N} Y_n\times\{n\}.
\end{equation*}
If $Y_n\subset[0,1)$ has more than $2\cdot 2^n$ points, then one of the $2^n$ intervals $I\in\mathcal D_n$ must contain at least three points of $Y_n$. But the restriction of $d$ to $I\times\{n\}$, and hence to $(Y_n\cap I)\times\{n\}$ is the Euclidean distance. Then the three collinear points of this set contradict the $\SRA(\alpha)$ condition. It follows that each $Y_n$ has at most finitely many points, in which case there must be infinitely many $n$ with $Y_n\neq\emptyset$. Let $n_1$, $n_2$, and $n_3$ be three such integers $n$, and let $y_i$ be a point in $Y_{n_i}$ for $i=1,2,3$. Then the restriction of $d$ to $\{(y_1,n_1),(y_2,n_2),(y_3,n_3)\}$ agrees with the Euclidean distance on the three-point set $\{n_1,n_2,n_3\}$, and this again contradicts the $\SRA(\alpha)$ property. We conclude that $X$ does not contain any infinite $\SRA(\alpha)$ set, so $X$ is not $\SRA(\alpha)$-full.

\section{Appendix B: A partial answer to Question \ref{Q:Hilbert}}\label{appendix-b}

For each $\tfrac1{\sqrt2} < \alpha < 1$, we provide an example of a subset $X \subset \ell^2$ with the property that $(X,||\cdot||_2)$ satisfies the $\SRA(\alpha)$ condition, but $(X,||\cdot||_2^q)$ fails to be a metric space for any $q>1$.

\smallskip

Fix any sequence $(\delta_m)$ with $\delta_m \le 1$ and $\delta_m \searrow 0$. Let $R>0$ be a large parameter (whose value will be specified later). We define a set $X \subset \ell^2$ as follows:
$$
X = \bigcup_{m=1}^\infty \{x_m,y_m,z_m\}, \qquad x_m = R \bolde_{2m-1}, \, y_m = R \bolde_{2m-1} + \bolde_{2m}, \, \mbox{and} \, z_m = R \bolde_{2m-1} + \delta_m \bw_m.
$$
Here $(\bolde_m)$ denotes the standard orthonormal basis of $\ell^2$, and $\bw_m$ is a unit vector in $\spa\{\bolde_{2m-1},\bolde_{2m}\}$ selected so that the triangle with vertices $x_m$, $y_m$, and $z_m$ has sides of length $1$, $\delta_m$, and $1+\alpha\delta_m$. 

Denote by $|\cdot|$ the $\ell^2$ norm and by $d$ the $\ell^2$ metric. The proof that $(X,d^q)$ fails to be a metric space for any $q>1$ is exactly the same as in Example \ref{ex:no-converse}. It suffices for us to show that $(X,d)$ satisfies the $\SRA(\alpha)$ condition. 

First, we provide an estimate for the distances between points $a,b \in X$ in the case when $a$ and $b$ lie in distinct two-dimensional subspaces. Let $a \in \{x_m,y_m,z_m\} \subset \spa\{ \bolde_{2m-1}, \bolde_{2m} \}$ and $b \in \{x_n,y_n,z_n\} \subset \spa\{ \bolde_{2n-1},\bolde_{2n} \}$, where $m \ne n$. Then
\begin{equation}\label{eq:d-a-b}
\sqrt2 R - \sqrt2 \le |a-b| = \sqrt{|a|^2+|b|^2} \le \sqrt2 R + \sqrt2.
\end{equation}

Now let $a,b,c$ be distinct points in $X$. Assume that $a \in \{x_m,y_m,z_m\}$, $b\in \{x_n,y_n,z_n\}$, and $c \in \{x_p,y_p,z_p\}$. 

\smallskip

\noindent {\bf Step 1.} If $m=n=p$ then $a,b,c$ form the vertices of a triangle in $\R^2$ with side lengths as above, and hence the $\SRA(\alpha)$ condition is satisfied. 

\smallskip

\noindent {\bf Step 2.} Suppose next that $m$, $n$, and $p$ are all distinct. We impose the restriction
\begin{equation}\label{eq:R-restriction-1}
R \ge 1 + \frac2\alpha.
\end{equation}
Then $\sqrt2 R + \sqrt2 \le (1+\alpha) (\sqrt2 R - \sqrt2)$ and hence, using \eqref{eq:d-a-b},
$$
|a-b| \le \max \{ |a-c| + \alpha |b-c|, \alpha |a-c| + |b-c| \}.
$$
and the corresponding statements for all permutations of the elements $a$, $b$, and $c$ also hold true. Hence in this case the $\SRA(\alpha)$ condition is satisfied.

\smallskip

\noindent {\bf Step 3.} Finally, we consider the case when $m=p \ne n$. Thus $a,c \in \{x_m,y_m,z_m\}$ and $b \in \{x_n,y_n,z_n\}$. We need to verify the following three inequalities:
\begin{equation}\label{eq:sra-alpha-1}
|a-c| \le \max \{ |a-b| + \alpha |b-c|, \alpha |a-b| + |b-c| \},
\end{equation}
\begin{equation}\label{eq:sra-alpha-2}
|a-b| \le \max \{ |a-c| + \alpha |c-b|, \alpha |a-c| + |c-b| \},
\end{equation}
and
\begin{equation}\label{eq:sra-alpha-3}
|c-b| \le \max \{ |a-b| + \alpha |a-c|, \alpha |a-b| + |a-c| \}.
\end{equation}
Note, however, that \eqref{eq:sra-alpha-2} and \eqref{eq:sra-alpha-3} are symmetric from the point of view of the assumptions on the points $a$, $b$, and $c$. We will in fact prove that
\begin{equation}\label{eq:sra-alpha-3-5}
\bigl| |a-b| - |c-b| \bigr| \le \alpha |a-c|
\end{equation}
which implies both \eqref{eq:sra-alpha-2} and \eqref{eq:sra-alpha-3}.

To verify \eqref{eq:sra-alpha-1} we observe that $|a-c| \le 1 + \alpha \delta_m$ since both $a$ and $c$ lie in $\{x_m,y_m,z_m\}$. On the other hand, $\min\{|a-b|,|c-b|\} \ge \sqrt2 R - \sqrt2$ by \eqref{eq:d-a-b}. Thus \eqref{eq:sra-alpha-1} holds provided
\begin{equation}\label{eq:sra-alpha-4}
1 + \alpha \delta_m \le (1+\alpha) (\sqrt2 R - \sqrt2).
\end{equation}
We impose the restriction
\begin{equation}\label{eq:R-restriction-2}
R \ge 1+ \frac1{\sqrt2}
\end{equation}
which implies that \eqref{eq:sra-alpha-4} holds true.

Finally, we turn to the proof of \eqref{eq:sra-alpha-3-5}.  
For convenience, we use an alternate representation for the points $a,b,c$. Write $a = R \bolde_{2m-1} + \vareps_a w_a$, $b = R \bolde_{2n-1} + \vareps_b w_b$, and $c = R \bolde_{2m-1} + \vareps w_c$. Here $\vareps_u \in \{0,1\}$ for $u=a,b,c$; this boolean variable indicates whether $u=x_m$ or $u \in \{y_m,z_m\}$. The variable $w_u$ for $u=a,b,c$ takes on one of the following values: $w_u \in \{\bolde_{2m},\delta_m \bw_m\}$ if $u=a,c$ and $w_u \in \{\bolde_{2n},\delta_n \bw_n\}$ if $u=b$. With this notation in place we compute
\begin{equation*}\begin{split}
|a-b| 
&= |R \bolde_{2m-1} + \vareps_a w_a - R \bolde_{2n-1} - \vareps_b w_b| \\
&= \sqrt{2R^2 + 2R \vareps_a \langle \bolde_{2m-1},w_a\rangle + 2 R \vareps_b \langle \bolde_{2n-1}, w_b \rangle + \vareps_a |w_a|^2 + \vareps_b |w_b|^2},
\end{split}\end{equation*}
$|a-c| = |\vareps_a w_a - \vareps_c w_c|$, and
\begin{equation*}\begin{split}
|c-b| 
&= |R \bolde_{2m-1} + \vareps_c w_c - R \bolde_{2n-1} - \vareps_b w_b| \\
&= \sqrt{2R^2 + 2R \vareps_c \langle \bolde_{2m-1},w_c\rangle + 2 R \vareps_b \langle \bolde_{2n-1}, w_b \rangle + \vareps_c |w_c|^2 + \vareps_b |w_b|^2}.
\end{split}\end{equation*}

If $\{a,c\} = \{x_m,y_m\}$ then $|a-c|=1$ and the desired equation \eqref{eq:sra-alpha-3-5} reads
\begin{equation}\label{eq:abc1}
\bigl| |a-b| - |c-b| \bigr| \le \alpha.
\end{equation}
Using the expressions for $|a-b|$ and $|c-b|$ above, along with the identity $|\sqrt{A} - \sqrt{B}| = \tfrac{|A-B|}{\sqrt{A} + \sqrt{B}}$, we rewrite \eqref{eq:abc1} as
\begin{equation}\begin{split}\label{eq:abc2}
1 &\le \alpha \left( \sqrt{2R^2 + 2R \vareps_a \langle \bolde_{2m-1},w_a\rangle + 2 R \vareps_b \langle \bolde_{2n-1}, w_b \rangle + \vareps_a |w_a|^2 + \vareps_b |w_b|^2} \right. \\
& \qquad \left. + \sqrt{2R^2 + 2R \vareps_c \langle \bolde_{2m-1},w_c\rangle + 2 R \vareps_b \langle \bolde_{2n-1}, w_b \rangle + \vareps_c |w_c|^2 + \vareps_b |w_b|^2} \right).
\end{split}\end{equation}
Note that $\vareps_u |\langle \bolde_{2m-1}, w_u \rangle| \le \delta_m$ for $u \in \{a,c\}$ and $\vareps_u |\langle \bolde_{2n-1}, w_u \rangle| \le \delta_n$ for $u=b$. Indeed, if the boolean indicator $\vareps_u = 0$ this is trivial, while if $\vareps_u = 1$ then either the two components of the inner product are orthogonal and again the result is trivial, or the expression in question is bounded above by $\delta_m |\bw_m| = \delta_m$ or $\delta_n |\bw_n| = \delta_n$. Hence \eqref{eq:abc2} is implied by
\begin{equation}\begin{split}\label{eq:abc3}
1 \le 2 \alpha \sqrt{2R^2 - 2R\delta_m - 2R\delta_n}.
\end{split}\end{equation}
Inequality \eqref{eq:abc3} in turn is implied by
$$
1 \le 2 \alpha \sqrt{2R^2-4R}
$$
which is equivalent to
\begin{equation}\label{eq:R-restriction-3}
R \ge 1 + \sqrt{1 + \frac1{8\alpha^2}}.
\end{equation}

If $\{a,c\} = \{x_m,z_m\}$ then $|a-c| = \delta_m$ and we wish to verify
\begin{equation}\label{eq:abc4}
\bigl| |a-b| - |c-b| \bigr| \le \alpha \delta_m
\end{equation}
which we rewrite as
\begin{equation}\begin{split}\label{eq:abc5}
& \delta_m | 2R \langle \bolde_{2m-1},\bw_m \rangle + \delta_m | \\
& \qquad \le \alpha \delta_m \left( \sqrt{2R^2 + 2R \vareps_a \langle \bolde_{2m-1},w_a\rangle + 2 R \vareps_b \langle \bolde_{2n-1}, w_b \rangle + \vareps_a |w_a|^2 + \vareps_b |w_b|^2} \right. \\
& \qquad \qquad \left. + \sqrt{2R^2 + 2R \vareps_c \langle \bolde_{2m-1},w_c\rangle + 2 R \vareps_b \langle \bolde_{2n-1}, w_b \rangle + \vareps_c |w_c|^2 + \vareps_b |w_b|^2} \right).
\end{split}\end{equation}
Using the bounds for $\vareps_u |\langle \bolde_{2m-1}, w_u \rangle|$ and $\vareps_u |\langle \bolde_{2n-1}, w_u \rangle|$ as discussed above, we see that \eqref{eq:abc5} is implied by
$$
2R+1 \le 2 \alpha \sqrt{2R^2-4R}
$$
which in turn is implied by 
$$
4(2\alpha^2-1) R^2 - 4(4\alpha^2+1) R - 1 \ge 0
$$
which is equivalent to
\begin{equation}\label{eq:R-restriction-4}
R \ge \frac{4\alpha^2+1 + \alpha \sqrt{16\alpha^2+10}}{4\alpha^2-2}.
\end{equation}

Finally, if $\{a,c\} = \{y_m,z_m\}$ then $|a-c| = 1 + \alpha \delta_m$ and we wish to verify
\begin{equation}\label{eq:abc6}
\bigl| |a-b| - |c-b| \bigr| \le \alpha (1+\alpha\delta_m)
\end{equation}
which we rewrite as
\begin{equation}\begin{split}\label{eq:abc7}
& | 1 - 2R \delta_m \langle \bolde_{2m-1},\bw_m \rangle - \delta_m^2 | \\
& \qquad \le \alpha (1+\alpha\delta_m) \left( \sqrt{2R^2 + 2R \vareps_a \langle \bolde_{2m-1},w_a\rangle + 2 R \vareps_b \langle \bolde_{2n-1}, w_b \rangle + \vareps_a |w_a|^2 + \vareps_b |w_b|^2} \right. \\
& \qquad \qquad \left. + \sqrt{2R^2 + 2R \vareps_c \langle \bolde_{2m-1},w_c\rangle + 2 R \vareps_b \langle \bolde_{2n-1}, w_b \rangle + \vareps_c |w_c|^2 + \vareps_b |w_b|^2} \right).
\end{split}\end{equation}
Following the line of reasoning as in previous cases, we see that \eqref{eq:abc7} is implied by
$$
1 + 2 R \delta_m + \delta_m^2 \le 2\alpha(1 + \alpha \delta_m) \sqrt{2R^2-4R}
$$
which in turn is implied by
$$
2 + 2R \le 2 \alpha \sqrt{2R^2-4R}
$$
which is equivalent to
\begin{equation}\label{eq:R-restriction-5}
R \ge \frac{2\alpha^2+1 + \alpha \sqrt{4\alpha^2+6}}{2\alpha^2-1}.
\end{equation}
If we assume that the parameter $R>0$ is chosen so large that \eqref{eq:R-restriction-1}, \eqref{eq:R-restriction-2}, \eqref{eq:R-restriction-3}, \eqref{eq:R-restriction-4} and  \eqref{eq:R-restriction-5} are satisfied, then we have that $(X,d)$ satisfies the $\SRA(\alpha)$ condition. (Note that the restriction $\alpha>1/\sqrt{2}$ is necessary in order to ensure the existence of a value $R$ satisfying \eqref{eq:R-restriction-4} and \eqref{eq:R-restriction-5}.) This finishes the proof of this example.


\end{document}